\documentclass[11pt, oneside]{amsart} 

\usepackage{geometry}     
\usepackage{comment}
\usepackage{graphicx}	
\usepackage{amsmath,amssymb,amsthm}
\usepackage{color}
\usepackage{pgfplots}
\usepackage{pgfplotstable}
\pgfplotsset{compat=newest}
\pgfplotsset{table/search path={PGFPlots}}
\usepgfplotslibrary{patchplots}
\usepgfplotslibrary{colormaps}
\usepgfplotslibrary{groupplots}
\usepackage{subcaption}
\usepackage{diagbox}

\usepackage{algorithm}
\usepackage{algorithmicx}
\usepackage{algpseudocode}%
\usepackage{hyperref}
\usepackage{url}

\renewcommand{\vec}[1]{\boldsymbol{#1}}

\newtheorem{theorem}{Theorem}
\newtheorem{proposition}[theorem]{Proposition}
\newtheorem{remark}[theorem]{Remark}%
\newtheorem{lemma}[theorem]{Lemma}%
\newtheorem{definition}[theorem]{Definition}
\newtheorem{assumption}[theorem]{Assumption}
\newtheorem{corollary}[theorem]{Corollary}

\newcommand{\R}{\mathbb{R}}

\newcommand{\N}{\mathbb{N}}

\newcommand{\magicC}{\kappa}
\DeclareMathOperator{\dist}{dist}

\renewcommand{\vec}[1]{\boldsymbol{\mathbf{#1}}}

\title{Adaptive RBF-multiscale approximation on bounded domains}
\author{Federico Lot and Christian Rieger} 
\address{Department of Mathematics and Computer Science (FB12),  Philipps-Universit\"at Marburg,
Hans-Meerwein-Stra\ss{}e 6, Marburg, 35032, Hessia, Germany}
\email{\{lotf,riegerc\}@mathematik.uni-marburg.de}
\thanks{This work is funded
by the Deutsche Forschungsgemeinschaft (DFG, German Research Foundation) –Projektnummer
452806809.}

\date{\today}

\begin{document}

\begin{abstract}
    This article addresses adaptivity in the kernel multiscale method. Adaptively compressed kernel multiscale approximations have already been presented and analyzed in \cite{LeGia:Wendland:2014}. 
    The main contribution of this work is to attempt to avoid using function evaluations which will be deleted in the compression step anyways. In order to work with function values, we always work in the Lagrange representation. We still assume to have all function values available to compute error norms, but we do not include those values in our approximations. Finally, we also employ local Lagrange function to further reduce the numerical work.
\end{abstract}

\maketitle


\section{Introduction}

The reconstruction of functions from finitely many (scattered) point evaluations has been a classical task in various applications. Historically, radial basis functions (RBFs) haven been proven to be useful in those applications. In order to improve numerical properties of the classical RBF interpolation, the so-called RBF-multiscale methods have been introduced in \cite{floater:iske:1996}. The analysis of these methods can be found in \cite{Wendland:2010} for bounded domains, in \cite{LeGia:etal:2010} for the sphere, and in the case of more general manifolds in \cite{Sharon:etal:2023}.
Recently, a variant of the classical multiscale algorithm which allows for better parallelization was presented and analyzed in \cite{Lot:Rieger:2026}.

The method is based on a hierarchy of sets of points and uses suitably scaled variants of the RBF to approximate the unknown function with an increasing amount of details on the finer levels. 

As already pointed out in \cite{LeGia:eta:2017}, the application might not need to use the same level of details everywhere. Sometimes, it might be enough to have only locally a detailed approximation of the unknown function. A somewhat complementary view is taken in \cite{LeGia:Wendland:2014}. Here, the authors study the situation where the quality of the approximation is locally already sufficient, and hence omit further details in those regions. This approach is in the spirit of adaptive methods but the adaption is made a posterior, i.e., after having computed all the details, it is decided which of them can be discarded given a certain error tolerance.

These works are the motivation for this article. The main difference from \cite{LeGia:Wendland:2014} lies in the selection or adaptation criteria. Indeed, we propose an alternative solution to remove the points prior to the computation of the local approximation, further reducing the computational cost, without compromising the approximation rate.
Technically, we base our analysis on a localization result from \cite{demarchi:wendland:2020}. This allows to safely remove points from the computation if the error is suitably small on a neighborhood of the point of interest. The details are presented in the novel algorithm \ref{alg:selection}. Based on our novel adaptation criterion, we propose the algorithm \ref{alg:adapt}, which is a variant of the classical multiscale algorithm of \cite{floater:iske:1996} enhanced with an adaptive point selection method. 
Moreover, in \ref{thm:adapt_errorbound}, we present an error analysis of our novel adaptive multiscale algorithm. The error bound is inspired by the bound in \cite{LeGia:Wendland:2014} and is explicit in the algorithm parameters. The explicit dependence of the problem parameters allows for suitable choices for the parameters. 
We further propose a variation of the main algorithm in order to improve its efficiency, inspired by locality of Lagrange function in \cite{Fuselier:etal:2013}. The error difference with our first proposed algorithm is analyzed in theorem \ref{thm:localerror}.

We point out that our method only requires knowledge of the unknown function of the finest possible grid (which has to be chosen in advance). In particular, we also need to have function evaluations at points which are discarded in the adaptive point selection. Compared to classical adaptive methods for the solution on operator equations, this seems too much information. But, we would like to stress that we do not have an operator equation to solve and hence no access to a right-hand side of this operator equation, which provides usually additional knowledge in adaptive methods.

The manuscript is structured as follows: In Section \ref{sec:notation}, we recall and extend results from the classical kernel multiscale method. 
Section \ref{sec:adaptivity} is the main section, where we introduce the new adaptive kernel multiscale algorithm and present an accompanying error analysis for it. Finally in Section \ref{sec:numerics}, we conclude this article by describing the numerical implementation and some numerical experiments.

\section{Notation and setting}
\label{sec:notation}
We first introduce the notation we use in the following document.

For scattered nodes, radial basis functions are a common tool to compute interpolations. Radial basis functions are described by a function $K_{\Phi}:\R^{d}\times \R^{d} \to \R$ with $K_{\Phi}(\vec{x},\vec{y}) =\Phi(\vec{x}-\vec{y})= \phi(\|\vec{x}-\vec{y} \|_2)$, where $\Phi:\R^{d} \to\R$ is a multivariate translation-invariant function and $\phi: \R \to \R$ is a univariate function. We are mostly interested in strictly positive definite functions with compact support.
We collect the assumptions on the kernel for later reference.
\begin{assumption}
	\label{ass:kernel}
	Let $\Phi:\R^{d} \to\R$ with $\Phi \in L_{1}(\R^d)$  be a radial function and $\tau>\frac{d}{2}$.  Let $c_{\Phi},C_{\Phi}>0$ be positive constants, and assume that the Fourier transform of $\Phi$ behaves like 
	\begin{equation} \label{eq:kernelfourier}
      c_{\Phi} \left(1+\|\vec{\omega}\|_{2}^{2}\right)^{-\tau} \le \widehat{\Phi}(\vec{\omega}) \le  C_{\Phi} \left(1+\|\vec{ \omega}\|_{2}^{2}\right)^{-\tau}  
    \end{equation}
	and that $\Phi$ has compact support in the unit ball, i.e.,
    \begin{equation*}
        \operatorname{supp}(\Phi) \subset B_{1}(0)=\{ \vec{x} \in \R^d : \| \vec{x}\|_2\le 1\}\subset \R^{d}.
    \end{equation*}
    The scaled kernel is defined as
	\begin{equation}\label{eq:kernelscaling}
    \Phi_{\delta}(\vec{x}) := \delta^{-d}\Phi(\vec{x}/\delta).
    \end{equation}
\end{assumption}

We also recall the definition of the norm for a functions $\Phi$ which assert Assumption \ref{ass:kernel},
\begin{equation}\label{eq:kernelnorm}
    \|g\|_{\Phi}^{2} := \int_{\R^{d}} \frac{|\widehat{g}(\boldsymbol{\omega})|^{2}}{\widehat{\Phi}(\boldsymbol{\omega})} d\boldsymbol{\omega}.
\end{equation}

Note that condition \eqref{eq:kernelfourier} ensures that $\Phi$ is the reproducing kernel of a Hilbert space 
\begin{equation*}
    \mathcal{H}_{\Phi}(\R^d) = \{g \in L^{2}(\R^d) \, : \, \|g\|_{\Phi} < \infty \}
\end{equation*}
which is isometric to the Sobolev space $H^{\tau}(\R^d)$, equipped with the alternative norm
\begin{equation*}
    \|g\|_{H^{\tau}(\R^{d})}^{2} := \int_{\R^{d}} |\widehat{g}(\boldsymbol{\omega})|^{2}(1+\|\boldsymbol{\omega}\|_{2}^{2})^{\tau} d\boldsymbol{\omega}.
\end{equation*}
Indeed, the two norms are equivalent. 

Additionally, in the scaling setting, the Fourier-transform of $\Phi_{\delta}$ is given as
\begin{equation}
     c_{\Phi} \left(1+\delta^{2}\|\vec{\omega}\|_{2}^{2}\right)^{-\tau} \le \widehat{\Phi_{\delta}}(\vec{\omega}) \le  C_{\Phi} \left(1+\delta^{2}\|\vec{\omega}\|_{2}^{2}\right)^{-\tau},
    \label{eq:deltafourier}
\end{equation}
which leads to the norm equivalence
\begin{equation*}
    c_{\Phi}^{1/2} \|g\|_{\Phi_{\delta}} \le \|g\|_{H^{\tau}(\R^{d})} \le C_{\Phi}^{1/2} \delta^{-\tau} \|g\|_{\Phi_{\delta}},
\end{equation*}
thus, the isometry of $\mathcal{H}_{\Phi_{\delta}}(\R^d)$ with $H^{\tau}(\R^d)$, for any choice of $\delta > 0$, see e.g. \cite{Wendland:2010}.

For multiscale methods, one usually considers not a fixed set of points but a sequence of sets of points.

As usual, for a given set $X \subset \Omega$, we have the fill distance
\begin{equation*}
   h_{X, \Omega} := \sup_{\vec{x} \in \Omega} \min_{\vec{x}_{j} \in X} \|\vec{x} - \vec{x}_{j}\|_{2}, 
\end{equation*}
and the separation distance
\begin{equation*}
     q_{X} := \frac{1}{2} \min_{j \neq k} \| \vec{x}_{j} - \vec{x}_{k}\|_{2}.
\end{equation*}

We collect the assumptions on the sets of points for later reference.
Moreover, let $N:\N=\{1,\dots\} \to \N$ denote a sequence of natural numbers and $X_{\ell} = \{\vec{x}^{(\ell)}_{1}, \dots, \vec{x}^{(\ell)}_{N(\ell)}\} \subset \Omega$ be a set of scattered nodes for $\ell \in \N$. We consider levels up to level $L\in \N$ and define the total number of points as $N=\sum_{\ell=1}^{L}N(\ell)$. We also assume the sets to be nested, i.e. $X_{\ell} \subset X_{\ell+1}$, although we point out results where such constraint can be lifted. 

Moreover, we define the usual discretization parameters $h_{\ell}:= h_{X_{\ell}, \Omega}$ and $q_{\ell}:=q_{X_{\ell}}$.

\begin{assumption}
    \label{ass:pointset}
    We assume $\Omega$ to be a bounded domain with Lipschitz boundary and $X_{1} \subset X_{2} \dots \subset X_{L} \subset \Omega$ for $L \in \N$.
    Moreover, we assume that there exist $k \ge 1$, $ \mu \in (0, k^{-1})$ and $c_{h} \in (0,1]$ such that
    \begin{align*}
       & c_{h} \mu h_{\ell} \le h_{\ell+1} \le \mu h_{\ell}, \quad \text{for } 1 \le \ell \le L \\
       & c_{h} \mu^{\ell} h_{0} \le h_{\ell} \le \mu^{\ell} h_{0}, \quad \text{for } 1 \le \ell \le L
    \end{align*}
    holds, where $h_{0} := h_{1}/\mu$.
    Next, we require that the sequence of the sets of points $X_{\ell}$ is quasi-uniform, i.e. 
    there exists a constant $c_{q}$ such that the following equivalence holds
    \begin{equation}\label{eq:quasiuniform}
        q_{\ell} \le h_{\ell} \le c_{q} q_{\ell},  \quad \quad \text{for } 1 \le \ell \le L.
    \end{equation}
    We further assume that the number of points at a given level is connected to the parameter $\mu$, i.e. that exists a constant $c_{\#}>0$ such that
    \begin{equation}\label{eq:Nmu}
        c_{\#} \mu^{-d\ell} \le N(\ell) \quad \quad \text{for } 1 \le \ell \le L
    \end{equation}
    holds. Finally, for fixed $\gamma>0$, for every $1 \le \ell \le L$, we choose 
    \begin{equation}\label{eq:deltadef}
    	\delta_{\ell} := \nu h_{\ell} \quad \text{ for }\nu \in \left[ \frac{1}{h_{1}} ,\frac{\gamma}{\mu} \right]
    \end{equation}
    as the scaling parameter in \eqref{eq:kernelscaling} on the  point set $X_{N(\ell)}$. 
\end{assumption}
It is worth pointing out that the introduction of $\gamma$, as well as the range of feasible values of $\nu$, come strictly from \cite{Wendland:2010}[Theorem 1]. Within the range, all choices of $\nu$ guarantee convergence; however, lower values are associated with sparser matrices, i.e., less computational time, while higher values are associated with higher accuracy.

Moreover, given \eqref{eq:quasiuniform}, there is a matching upper bound for the inequality \eqref{eq:Nmu}. Indeed, a volume comparison argument gives 
\begin{equation}\label{eq:Nmu2}
    N(\ell) \le C_{\#} \mu^{-d\ell}.
\end{equation}

Furthermore, we have 
\begin{lemma}\label{lem:Lstar}
	Let Assumption \ref{ass:pointset} holds. Additionally, without loss of generality, assume $h_{\ell} < 1$, $\ell \in \N_{0}$ (since we refer to bounded domains, we can always rescale ourselves in the unitary box).
	Then, we get for $\ell \ge L^{\star}:=\lceil \frac{k\mu}{1-k\mu} \rceil$ the bound
	 \begin{align*}
	 	k h_{\ell+1} |\log(h_{\ell+1}) | \le h_{\ell} |\log(h_{\ell}) |.
	\end{align*}
\end{lemma}
\begin{proof}
    Since both $h_{\ell}$ and $\mu$ are smaller than $1$, we have
    \begin{align*}
    	k h_{\ell+1} |\log(h_{\ell+1}) | = k \mu h_{\ell} \left( |\log(h_{\ell})| + |\log(\mu)|\right).
    \end{align*}
    Thus, is sufficient to require
    \begin{align*}
    	&k \mu \left( |\log(h_{\ell})| + |\log(\mu)|\right) \le |\log(h_{\ell})| \\
        &\iff  |\log(h_{\ell})|(1-k\mu) \ge k\mu |\log(\mu)| \\
        &\iff  |\log(h_{\ell})| \ge \frac{k\mu}{(1-k\mu)} |\log(\mu)|.
    \end{align*}
    Here we can directly relate the result to $\ell$, exploiting our assumption, in particular, $h_{\ell} < 1$ for every $\ell$, we have
    \begin{equation*}
        |\log(h_{\ell})| \ge |\log(h_{0}\mu^{\ell})| \ge \ell|\log(\mu)|
    \end{equation*} 
    which for $\ell \ge \lceil\frac{k\mu}{(1-k\mu)}\rceil$ leads to the desired upper bound.
\end{proof}

\begin{remark}
    As a practical observation on the stated lemma we have, e.g. for $k=1$, $\mu=1/2$ that this result holds for $\ell \ge 2$, and further examples are summarized in Table \ref{tab:hloghbound}.
    \begin{table}
        \centering
        \bgroup
        \begin{tabular}{c|c|c|c|c|c|c}
             $k$ & $1$ & $1$ & $1$ & $5$ & $5$ & $5$\\
             $\mu$ & $0.5$ & $0.9$ & $0.1$ & $0.1$ & $0.19$ & $0.01$ \\
             \hline 
             \noalign{\smallskip} $ L^{\star}$ & $1$ & $9$ & $1$ & $1$ & $19$ & $1$\\ 
        \end{tabular}
        \egroup
        \caption{upper bound $L^{\star}$ with different choices of $k$ and $\mu$}
        \label{tab:hloghbound}
    \end{table}
\end{remark}

The preliminary lemma \ref{lem:Lstar} highlights that, from level $\ell > L^{\star}$, we have that 
\begin{equation*}
    B_{k h_{\ell+1} |\log(h_{\ell+1}) |}(\vec{x}) \subseteq B_{h_{\ell} |\log(h_{\ell}) |}(\vec{x}), \qquad x \in \R^{d},
\end{equation*}
which we exploit in Theorem \ref{thm:consistency}.

Next, introducing the notation $\Phi_{\ell}:=\Phi_{\delta_{\ell}}$, we define approximation spaces on the level $1\le \ell \le L$ as
\begin{equation}\label{eq:localspaces}
    W_{\ell} = \text{span}\{\Phi_{\ell}(\cdot, \vec{x}_{j}^{(\ell)}) \, : \, 1\le j \le N(\ell) \}\subset \mathcal{H}_{\Phi_{\ell}}(\R^d),
\end{equation}
where $\mathcal{H}_{\Phi_{\ell}}(\R^d)$ is the reproducing kernel Hilbert space of the scaled translation-invariant kernel $\Phi_{\ell}$. Precisely, $\mathcal{H}_{\Phi_{\ell}}(\R^d)$ is the completion of $\{\Phi_{\ell}(\cdot, \vec{x}_{j}^{(\ell)}) \, : \, 1\le j \le N(\ell) \}$ with respect to the kernel-norm $\|\cdot\|_{\Phi_{\ell}}$, \cite{wendland_2004}.

The global approximation space is given as 
\begin{equation*}
	V_{L} = W_{1} + \dots +W_{L},
\end{equation*}
where the sum is in general not direct.

We define the interpolator operator $\mathcal{I}_{\ell}:C(\Omega)\to \mathcal{H}_{\Phi_{\ell}}$ based on the kernel $\Phi_{\ell}$. This operator is given as
\begin{equation}
    \label{eq:approximant}
 \mathcal{I}_{\ell}(f)= \sum_{n=1}^{N(\ell)} \alpha_{n}^{(\ell)} \Phi_{\ell}(\cdot, \vec{x}_{n}^{(\ell)}),
\end{equation}
where, given 
\begin{equation*}
    A_{h, \ell} := \left( \Phi_{\ell}(\vec{x}^{(h)}, \vec{x}^{(\ell)}) \right)_{\genfrac{}{}{0pt}{}{\vec{x}^{(h)} \in X_{h}}{\vec{x}^{(\ell)} \in X_{\ell}}} \in \R^{N(h)\times N(\ell)},
\end{equation*}
the coefficients $\vec{\alpha}^{(\ell)} \in \R^{N(\ell)}$ satisfy the system of linear equations
\begin{equation*}
\ A_{\ell, \ell} \vec{\alpha}^{(\ell)} = \mathbf{f}^{(\ell)}. 
\end{equation*}

Assuming that we are given the sequence of sets of points $X_{\ell}$ for $1\le \ell \le L$, the approximation of a function $f$ is computed, following the idea of \cite{floater:iske:1996}. Initially, we set $f_0 = 0$ and $e_0 = f$. 
For $1\le \ell \le L$, we iterate
\begin{itemize}
	\item compute approximand $s_{\ell} (f):= \mathcal{I}_{\ell}(e_{\ell-1}) \in W_{\ell}$ based on point set $X_{N(\ell)}$
	\item update $f_{\ell} = f_{\ell-1} + s_{\ell}$ and  $e_{\ell} = e_{\ell-1} - s_{\ell}$ 
\end{itemize} 	
to obtain the final approximation $\sum_{\ell=1}^{L}s_{\ell} (f) \in V_{L}$ with $\sum_{\ell=1}^{L}s_{\ell} (f) \approx f$.
From the algorithm, we directly obtain the recursive representation
\begin{equation}\label{eq:recursive:localinterpol}
	 s_{\ell}(f) = \mathcal{I}_{\ell}(e_{\ell-1}) = \mathcal{I}_{\ell}\left(f - \sum_{j = 1}^{\ell-1} s_{j}(f)\right), \quad 1\le \ell \le L,
\end{equation} 
At the level of coefficients, iteration \eqref{eq:recursive:localinterpol} corresponds to 
\begin{equation}\label{eq:mas}
    A_{\ell, \ell} \vec{\alpha}^{(\ell)} = \mathbf{f}^{(\ell)} - \sum_{k=1}^{\ell-1} A_{\ell k} \vec{\alpha}^{(k)}, \quad \mathbf{f}^{(\ell)}: = f|_{X_{\ell}} \in \R^{N(\ell)}.
\end{equation}

The methods can also be formulated as an algorithm given in Alg. \ref{alg:ms}.

\begin{algorithm}[ht]
    \caption{Multiscale approximation \label{alg:ms}}
    \begin{algorithmic}[1]
    
    \State{Input: Number of levels $L$, sets of points $X_{1}, \dots, X_{L}$, rhs $f$}
    \State{Output: Multiscale approximation $f_{L}$.}
    \vspace{0.1cm}
    \State $e_{0} \leftarrow f$; 
    \State $f_{0} \leftarrow 0$; 
    \For{$\ell = 1, \dots, L-1$}   
        \State $s_{\ell} \leftarrow \mathcal{I}_{\ell}(e_{\ell-1})$; 
        \State $e_{\ell} \leftarrow e_{\ell-1} - s_{\ell}$; 
        \State $f_{\ell} \leftarrow f_{\ell-1} + s_{\ell}$; 
    \EndFor
    \end{algorithmic}
\end{algorithm}

From \cite[Theorem 1]{Wendland:2010}, we have an error bound for the approximation using algorithm \ref{alg:ms}.

Since we assume our target function to be defined on a bounded domain with Lipschitz boundary, an extension operator can be defined.
We need the following result from \cite[Proposition 1]{Wendland:2010}
\begin{proposition} \label{prop:extension}
    Suppose $\Omega \subseteq \R^{d}$ is open and has a Lipschitz boundary. Let $\tau \ge 0$. Then, it exists a linear operator $E: H^{\tau}(\Omega) \rightarrow H^{\tau}(\R^{d})$, such that, for all $f\in H^{\tau}(\Omega)$,
    \begin{align*}
        &Ef|_{\Omega} = f|_{\Omega}, \\
        &\|Ef\|_{H^{\tau}(\R^{d})} \le C_{\tau} \|f\|_{H^{\tau}(\Omega)},
    \end{align*}
    i.e., $E$ is a bounded extension operator. Furthermore, the same operator $E$ can be used for every $\tau \ge 0$.
\end{proposition}
This allows us to have
\begin{equation*}
    \|f\|_{H^{\tau}(\Omega)} \le \|Ef\|_{H^{\tau}(\R^{d})} \le C_{\tau} \|f\|_{H^{\tau}(\Omega)},
\end{equation*}
i.e., defining a norm on ${H^{\tau}(\Omega)}$ via the extension operator. 
Finally, we can state
\begin{theorem}\label{thm:errorbound}
    Let $\Omega \subseteq \R^{d}$ be a bounded domain with Lipschitz boundary and $f \in H^{\tau}(\Omega)$. Under assumptions \ref{ass:pointset} and \ref{ass:kernel}, there exists a constant $C_{1}=C_{1}(\gamma)$ such that 
    \begin{equation}
       \|Ee_{j}\|_{\Phi_{j+1}} \le \alpha\|Ee_{j-1}\|_{\Phi_{j}}
    \end{equation}
    with $\alpha = C_{1}\mu^{\tau}$. Moreover,
    \begin{equation*}
        \left\| f - \sum_{\ell=1}^{L}s_{\ell} (f) \right\|_{L_{2}}\le C_{2}\left(C_{1}  \mu^{\tau} \right)^{L} \|f\|_{H^{\tau}(\Omega)}
    \end{equation*}
    for all $L\in \N$.  Thus, the multiscale approximation converges linearly in the $L_{2}$ norm if $C_{1} \mu^{\tau} < 1$.
\end{theorem}

Additionally, we have

\begin{lemma}\label{lem:classic-infbound}
    Let $\Omega \subseteq \R^{d}$ be a bounded domain with Lipschitz boundary and $f \in H^{\tau}(\Omega)$. Under assumptions \ref{ass:pointset} and \ref{ass:kernel}, for all $L\in \N$, we have
    \begin{equation}\label{eq:classic-infbound}
        \| e_{\ell} (f) \|_{L_{\infty}(\Omega)} \le C_{2}(C_{1}  \mu^{\tau-d/2} )^{L} \|f\|_{H^{\tau}(\Omega)},
    \end{equation}
    where constants $C_{1}, C_{2}$ stems from \ref{thm:errorbound}.  
\end{lemma}
\begin{proof}
    Due to the Sobolev embedding theorem and the zero-point estimate, we have
    \begin{align*}
        \| e_{\ell} (f) \|_{L_{\infty}(\Omega)}& \le C_{\Omega} \| e_{\ell} (f) \|_{H^{\iota}(\Omega)} \\
        & \le C_{\Omega} h_{\ell}^{\tau-\iota} \| e_{\ell} (f) \|_{H^{\tau}(\Omega)}
    \end{align*}
    for every $\tau > \iota > d/2$. From here, standard argument from \cite{Wendland:2010} leads to
    \begin{align*}
        \| e_{L} (f) &\|_{L_{\infty}(\Omega)} \le C_{\Omega} h_{L}^{\tau-\iota} \| e_{L} (f) \|_{H^{\iota}(\Omega)} \\
        & \le C_{\Omega} h_{L}^{\tau-\iota} \| Ee_{L} (f) \|_{H^{\iota}(\R^{d})} \\
        & \le C_{\Omega} C_{\Phi}^{1/2} h_{L}^{\tau-\iota} \delta_{L+1}^{-\tau}\| Ee_{L} (f) \|_{\Phi_{L+1}} \\
        & \le C_{\Omega} C_{\Phi}^{1/2} h_{L}^{\tau-\iota} \delta_{L+1}^{-\tau} (C_{1}\mu^{\tau})^{L} \| Ef \|_{\Phi_{1}} \\
        & \le C_{\Omega} (C_{\Phi}/c_{\Phi})^{1/2} h_{L}^{\tau-\iota} \delta_{L+1}^{-\tau} (C_{1}\mu^{\tau})^{L} \| Ef \|_{H^{\tau}(\R^{d})} \\
        & \le c_{\ref{eq:classic-infbound}} (C_{1}\mu^{\tau-\iota-1/L})^{L} \| Ef \|_{H^{\tau}(\R^{d})}  \le c_{\ref{eq:classic-infbound}} (C_{1}\mu^{\tau-d/2})^{L} \| f \|_{H^{\tau}(\Omega)},
    \end{align*}
    where we collected constants with $c_{\ref{eq:classic-infbound}} = C_{\Omega} (C_{\Phi}/c_{\Phi})^{1/2} c_{h}^{-\tau} h_{0}^{-\iota}$. 
\end{proof}

We recall that in our settings, we have $2\tau = d+2k+1$, which implies that $\tau  > d/2$ for every choice of Wendland function. 
We further introduce the cardinal functions 
\begin{align}\label{eq:cardinal}
\chi^{(\ell)}_{n} \in W_{\ell} \quad \text{s.t. }\chi^{(\ell)}_{n}(\vec{x}^{(\ell)}_{m}) = \delta_{n,m} \quad \forall \, 1\le n,m\le N(\ell).
\end{align} 
We observe the identity
\begin{equation*}
A_{\ell, \ell} 
\left(\begin{array}{c}
    \chi_{1}^{(\ell)}(\vec{x}) \\
    \chi_{2}^{(\ell)}(\vec{x}) \\
    \vdots \\
    \chi_{N(\ell)}^{(\ell)}(\vec{x})
\end{array}\right) = 
\left(\begin{array}{c}
    \Phi_{\ell}(\vec{x}, \vec{x}^{(\ell)}_{1}) \\
    \Phi_{\ell}(\vec{x}, \vec{x}^{(\ell)}_{2}) \\
    \vdots \\
    \Phi_{\ell}(\vec{x}, \vec{x}^{(\ell)}_{N(\ell)})
\end{array}\right) ,
\end{equation*}
from which it follows that
\begin{equation}
    \chi_{i}^{(\ell)}(\vec{x}) = \sum\nolimits_{k=1}\nolimits^{N(\ell)} (A^{-1}_{\ell, \ell})_{ik} \Phi_{\ell}(\vec{x}, \vec{x}^{(\ell)}_{k}) = \sum\nolimits_{k: \|\vec{x}-\vec{x}^{(\ell)}_{k}\|_{2} \le \delta_{\ell}} 
    (A^{-1}_{\ell, \ell})_{ik} \Phi_{\ell}(\vec{x}, \vec{x}^{(\ell)}_{k})
    \label{eq:chi}
\end{equation}
holds for all $\vec{x} \in \R^{d}$.
We can then recast \eqref{eq:approximant} in the following form
\begin{equation}
    s_{\ell} (f)= \sum_{n=1}^{N(\ell)} e^{(\ell)}(\vec{x}_{n}^{(\ell)}) \chi^{(\ell)}_{n}.
\end{equation}
Since our motivation is to leverage the proprieties of the cardinals, we make use of this formulation frequently.
We also need the following result from \cite[Lemma 8]{Lot:Rieger:2026}:

\begin{lemma}\label{lem:lagrangedecay}
    Let the usual assumptions on $\Phi_{\ell}$ and $X_{\ell}$ be satisfied. Then, there is $\theta>0$ such that for all $1 \le i,j \le N(\ell))$ and all $1\le \ell \le L$, the bound
    \begin{equation}\label{eq:lagrangedecayA}
        \left|\left(A_{\ell, \ell}^{-1}\right)_{ij}\right| \le 2(C_{d}c_{\Phi})^{-1} \left(1 +4 M_{d}^{2} \nu^{2} c_{q}^{2}\right)^{\tau} e^{2 \theta \sqrt{d}} q_{\ell}^{d} e^{-\theta \|\vec{x}_{i}^{(\ell)}-\vec{x}_{j}^{(\ell)}\|_{2}/q_{\ell}}
    \end{equation}
    holds.
    Moreover, it holds that
    \begin{equation}\label{eq:lagrangedecay}
        \left| \chi_{i}^{(\ell)}(\vec{y})\right| \le 2(C_{d}c_{\Phi})^{-1} \left(1 +4 M_{d}^{2} \nu^{2} c_{q}^{2}\right)^{\tau} e^{2 \theta \sqrt{d}} e^{c_{q} \nu \theta} (1+\nu c_{q})^{d}e^{-\theta\|\vec{x}_{i}^{(\ell)}-\vec{y}\|_{2}/q_{\ell}},
    \end{equation}
    where the constants $C_{d},M_d$ come from the bounds on the condition number of $A_{\ell,\ell}$.
\end{lemma}

We compact the constant of those results to achieve a clean representation, where the interested reader can always check back to the original result their details. Thus, we introduce 
\begin{align*}
    c_{\ref{eq:lagrangedecayA}} & = 2(C_{d}c_{\Phi})^{-1} \left(1 +4 M_{d}^{2} \nu^{2} c_{q}^{2}\right)^{\tau} e^{2 \theta \sqrt{d}}, \\
    c_{\ref{eq:lagrangedecay}} & = 2(C_{d}c_{\Phi})^{-1} \left(1 +4 M_{d}^{2} \nu^{2} c_{q}^{2}\right)^{\tau} e^{2 \theta \sqrt{d}} e^{c_{q} \nu \theta} (1+\nu c_{q})^{d}.
\end{align*}

Following the general pattern of \cite{Hangelbroek:etal:2010,Hangelbroek:etal:2011}, we can derive a bound on the Lebesgue constant.  See also \cite[Cor. 2.5]{demarchi:wendland:2020} for a bound on the Lebesgue constant for compactly supported kernels. 
We carry out the computations in order to keep track of the multiple involved constants.
\begin{corollary}\label{cor:lagrangedecaysum}
    Following the hypothesis of Lemma \ref{lem:lagrangedecay}, we have
    \begin{equation} \label{eq:lagrangedecayAsum}
        \sum_{i,j = 1}^{N(\ell)} \left|\left(A_{\ell, \ell}^{-1}\right)_{ij}\right| \le c_{\ref{eq:lagrangedecayAsum}},
    \end{equation}
    \begin{equation} \label{eq:lagrangedecaysum}
        \sum_{i = 1}^{N(\ell)} \left|\chi^{(\ell)}_{i}(\vec{y})\right| \le c_{\ref{eq:lagrangedecaysum}},  \qquad \vec{y} \in \R^{d},
    \end{equation}
    for every $1 \le \ell \le L$.
\end{corollary}
\begin{proof}
    Defining 
    \begin{equation*}
		\mathcal{A}^{(\ell)}_{\vec{x}}(m):= B_{(m+1) q_{\ell}}(\vec{x})\setminus B_{m q_{\ell}}(\vec{x}),
	\end{equation*} 
    we can exploit the annuli-sum trick (\cite{Narcowich:Ward:1991}) and Equation \ref{eq:lagrangedecayA}, to compute
    \begin{align*}
        \sum_{i,j = 1}^{N(\ell)} \left|\left(A_{\ell, \ell}^{-1}\right)_{ij}\right| & \le \sum_{i,j = 1}^{N(\ell)} c_{\ref{eq:lagrangedecayA}} q_{\ell}^{d} e^{-\theta \|\vec{x}^{(\ell)}_{i}-\vec{x}^{(\ell)}_{j}\|_{2}/q_{\ell}} \\
        & \le c_{\ref{eq:lagrangedecayA}} q_{\ell}^{d} \sum_{i = 1}^{N(\ell)} \sum_{m=0}^{\infty} \sum_{\vec{x}^{(\ell)}_{j} \in X_{\ell} \cap \mathcal{A}^{(\ell)}_{\vec{x}^{(\ell)}_{i}}(m)} e^{-\theta \|\vec{x}^{(\ell)}_{i}-\vec{x}^{(\ell)}_{j}\|_{2}/q_{\ell}} \\
         &\le c_{\ref{eq:lagrangedecayA}} q_{\ell}^{d} \sum_{i = 1}^{N(\ell)} \sum_{m=0}^{\infty} \#(X_{\ell} \cap \mathcal{A}^{(\ell)}_{\vec{x}^{(\ell)}_{i}}(m)) e^{-\theta m},
    \end{align*}
    where using a volume comparison argument to estimate the number of points, we get
    \begin{align*}
       & \sum_{i,j = 1}^{N(\ell)} \left|\left(A_{\ell, \ell}^{-1}\right)_{ij}\right|  \le c_{\ref{eq:lagrangedecayA}} q_{\ell}^{d} \sum_{i = 1}^{N(\ell)} \sum_{m=0}^{\infty} (1+(m+1)q_{\ell}/q_{\ell})^{d} e^{-\theta m} \\
       & \le c_{\ref{eq:lagrangedecayA}} q_{\ell}^{d} \#{N(\ell)} \sum_{m=0}^{\infty} (2+m)^{d} e^{-\theta m} 
         \le c_{\ref{eq:lagrangedecayA}}N(1)(h_{1}/q_{\ell})^{d}  q_{\ell}^{d} c_{\Sigma}
         = c_{\ref{eq:lagrangedecayA}}N(1)h_{1}^{d} c_{\Sigma} = c_{\ref{eq:lagrangedecayAsum}}.
    \end{align*}
    Similarly, we have 
    \begin{align*}
       & \sum_{i = 1}^{N(\ell)} \left|\chi^{(\ell)}_{i}(\vec{y})\right|  \le \sum_{i = 1}^{N(\ell)} c_{\ref{eq:lagrangedecay}} e^{-\theta\|\vec{x}_{i}^{(\ell)}-\vec{y}\|_{2}/q_{\ell}} 
         \le c_{\ref{eq:lagrangedecay}} \sum_{m=0}^{\infty} \sum_{\vec{x}^{(\ell)}_{i} \in X_{\ell} \cap \mathcal{A}^{(\ell)}_{\vec{y}}(m)} e^{-\theta \|\vec{x}^{(\ell)}_{i}-\vec{y}\|_{2}/q_{\ell}} \\
        & \le c_{\ref{eq:lagrangedecay}} \sum_{m=0}^{\infty} (2+m)^{d} e^{-\theta m} = c_{\ref{eq:lagrangedecaysum}}
    \end{align*}
\end{proof}

We now introduce a variation for the cardinal function that has a lower computational cost.

In the remainder of the article, we use the notation 
\begin{equation}\label{eq:Bellxr}
    B^{(\ell)}(\vec{x}; r) := X_{\ell} \cap B(\vec{x}; r) = \{\vec{x}^{(\ell)} \in X_{\ell} \, : \, \|\vec{x}^{(\ell)}-\vec{x}\|_{2} \le r \} \subseteq X_{\ell}.
\end{equation}
Following, e.g., \cite{Fuselier:etal:2013} on the sphere, we introduce the following variations of the Lagrange basis.
Again, we do the computations to keep track of the involved constants.

\begin{lemma}\label{lem:localcardinal-inf}
    Let $\check{\chi}^{(\ell)}_{i}$ be the cardinal function where the Lagrange conditions are enforced only in $B^{(\ell)}(\vec{x}_{i}^{(\ell)}; \varrho h_{\ell} |\log h_{\ell}|)$. Then,
    \begin{equation}\label{eq:localcardinal-inf}
        \|\chi^{(\ell)}_{i}-\check{\chi}^{(\ell)}_{i}\|_{L_{\infty}(\Omega)} \le c_{\ref{eq:localcardinal-inf}} h_{\ell}^{\theta \varrho-d}
    \end{equation}
    holds.
\end{lemma}
\begin{proof}
    We proceed as in the previous proof, introducing the notation for the coefficients of the cardinal functions. Let $A := (A_{\ell, \ell})^{-1}$.
    Let $\chi^{(\ell)}_{i}$ be the standard cardinal function based on $\Phi_{\ell}$ on the set of points $X_{\ell}$, i.e.
    \begin{equation*}
        \chi^{(\ell)}_{i} = \sum_{j=1}^{N(\ell)} A_{i j} \Phi_{\ell}(\vec{x}_{j}^{(\ell)}, \cdot),
    \end{equation*}
    and let $\mathcal{I}_{i}^{(\ell)} := \{\vec{x}^{(\ell)} \in X_{\ell} \, : \, \|\vec{x}_{i}^{(\ell)}-\vec{x}^{(\ell)}\|_{2} \le \varrho h_{\ell}|\log(h_{\ell})|\} \subset X_{\ell}$ and $\mathcal{I}_{i}^{(\ell)\complement} := X_{\ell} \setminus \mathcal{I}_{i}^{(\ell)}$.  Then, we define 
    \begin{align*}
        \check{\chi}^{(\ell)}_{i} &:= \sum_{\vec{x}^{(\ell)} \in \mathcal{I}_{i}^{(\ell)}} B_{i j} \Phi_{\ell}(\vec{x}^{(\ell)}, \cdot) \\
    \end{align*}
    where, given $B_{\ell}(i) := \left(\Phi_{\ell}(\vec{x}^{\ell}_{j},\vec{x}^{\ell}_{k}) \right)_{\vec{x}^{\ell}_{j},\vec{x}^{\ell}_{k} \in \mathcal{I}_{i}^{(\ell)}}$, we define $B := (B_{\ell}(i))^{-1}$.
    We have
    \begin{equation*}
        A_{\ell,\ell} = \left(\begin{array}{cc}
           B_{\ell}(i) & C_{\ell}(i) \\
           C_{\ell}(i)^{T}  & D_{\ell}(i)
        \end{array}\right), 
    \end{equation*}
    with 
    \begin{equation*}
        C_{\ell}(i) := \left(\Phi_{\ell}(\vec{x}^{\ell}_{j},\vec{x}^{\ell}_{k}) \right)_{\vec{x}^{\ell}_{j} \in \mathcal{I}_{i}^{(\ell)},\vec{x}^{\ell}_{k} \in \mathcal{I}_{i}^{(\ell)\complement}} \quad \text{ and } \quad D_{\ell}(i) := \left(\Phi_{\ell}(\vec{x}^{\ell}_{j},\vec{x}^{\ell}_{k}) \right)_{\vec{x}^{\ell}_{j},\vec{x}^{\ell}_{k} \in \mathcal{I}_{i}^{(\ell)\complement}}.
    \end{equation*}
    Thus, given the invertibility of $B_{\ell}(i)$, we can highlight the structure of the matrix $A$. Indeed, 
    \begin{equation*}
        A = (A_{\ell,\ell})^{-1} = \left(\begin{array}{cc}
           B-MC_{\ell}(i)^{T}B & M  \\
           \star  & \star
        \end{array}\right), 
    \end{equation*}
    with $M:= BC_{\ell}(i)(D_{\ell}(i)-C_{\ell}(i)^{T}BC_{\ell}(i))^{-1}$.
    Then, 
    \begin{align*}
        \|\chi_{i}-\check{\chi}_{i}\|_{L_{\infty}(\Omega)} & = \left\|\sum_{\vec{x}^{(\ell)}_{j} \in X_{\ell}} A_{ij}\Phi_{\ell}(\vec{x}^{(\ell)}_{j}, \cdot) - \sum_{\vec{x}^{(\ell)}_{j} \in \mathcal{I}_{i}^{(\ell)}} B_{ij}\Phi_{\ell}(\vec{x}^{(\ell)}_{j}, \cdot)\right\|_{L_{\infty}(\Omega)} \\
        & \le \left\|\sum_{\vec{x}^{(\ell)}_{j} \in \mathcal{I}_{i}^{(\ell)}} (A-B)_{ij}\Phi_{\ell}(\vec{x}^{(\ell)}_{j}, \cdot)\right\|_{L_{\infty}(\Omega)} + \left\| \sum_{\vec{x}^{(\ell)}_{j} \in \mathcal{I}_{i}^{(\ell)\complement}} A_{ij}\Phi_{\ell}(\vec{x}^{(\ell)}_{j}, \cdot)\right\|_{L_{\infty}(\Omega)} \\
        & = \left\|\sum_{\vec{x}^{(\ell)}_{j}, \vec{x}^{(\ell)}_{k} \in \mathcal{I}_{i}^{(\ell)}, \vec{x}^{(\ell)}_{h} \in \mathcal{I}_{i}^{(\ell)\complement}} M_{ih} C_{\ell}(i)^{T}_{hk}B_{kj}\Phi_{\ell}(\vec{x}^{(\ell)}_{j}, \cdot)\right\|_{L_{\infty}(\Omega)} + \\
        &+\left\| \sum_{\vec{x}^{(\ell)}_{j} \in \mathcal{I}_{i}^{(\ell)\complement}} A_{ij}\Phi_{\ell}(\vec{x}^{(\ell)}_{j}, \cdot)\right\|_{L_{\infty}(\Omega)}.
    \end{align*}
    From here, noting that $M$ is a submatrix of $A$ we can apply equation \ref{eq:lagrangedecayA} and bound
    \begin{align*}
        &\left\|\sum_{\vec{x}^{(\ell)}_{j}, \vec{x}^{(\ell)}_{k} \in \mathcal{I}_{i}^{(\ell)}, \vec{x}^{(\ell)}_{h} \in \mathcal{I}_{i}^{(\ell)\complement}} M_{ih} C_{\ell}(i)^{T}_{hk}B_{kj}\Phi_{\ell}(\vec{x}^{(\ell)}_{j}, \cdot)\right\|_{L_{\infty}(\Omega)} \\
        &\le \sum_{\vec{x}^{(\ell)}_{h} \in \mathcal{I}_{i}^{(\ell)\complement}} c_{\ref{eq:lagrangedecayA}} q_{\ell}^{d} e^{-\theta\|\vec{x}^{(\ell)}_{i}-\vec{x}^{(\ell)}_{h}\|_{2}/q_{\ell}} \max_{\vec{x}_{k}^{(\ell)} \in \mathcal{I}_{i}^{(\ell)}}|\Phi_{\ell}(\vec{x}_{k}^{(\ell)}, \vec{x}_{h}^{(\ell)})| \| \sum_{\vec{x}^{(\ell)}_{k} \in \mathcal{I}_{i}^{(\ell)}}\chi^{(\mathcal{I}_{i}^{(\ell)})}_{k}\|_{L_{\infty}(\Omega)} \\
        & \le c_{\ref{eq:lagrangedecayA}} c_{\ref{eq:lagrangedecaysum}} (q_{\ell}/\delta_{\ell})^{d} h_{\ell}^{\theta \varrho} \#(\mathcal{I}_{i}^{(\ell)\complement}).
    \end{align*}
    Also
    \begin{align*}
        \left\| \sum_{\vec{x}^{(\ell)}_{j} \in \mathcal{I}_{i}^{(\ell)\complement}} A_{ij}\Phi_{\ell}(\vec{x}^{(\ell)}_{j}, \cdot)\right\|_{L_{\infty}(\Omega)} & \le \sum_{\vec{x}^{(\ell)}_{j} \in \mathcal{I}_{i}^{(\ell)\complement}} c_{\ref{eq:lagrangedecayA}} q_{\ell}^{d} e^{-\theta\|\vec{x}^{(\ell)}_{i}-\vec{x}^{(\ell)}_{j}\|_{2}/q_{\ell}} \|\Phi_{\ell}(\vec{x}^{(\ell)}_{j}, \cdot)\|_{L_{\infty}(\Omega)} \\
        & \le c_{\ref{eq:lagrangedecayA}} \nu^{-d} h_{\ell}^{\theta \varrho-d} N(0)c_{q}^{d}.
    \end{align*}
    Putting everything together, we get
    \begin{equation*}
        \|\chi_{i}-\check{\chi}_{i}\|_{L_{\infty}(\Omega)} \le c_{\ref{eq:lagrangedecayA}} \nu^{-d} h_{\ell}^{\theta \varrho-d} N(0)c_{q}^{d}(c_{\ref{eq:lagrangedecaysum}}+1),
    \end{equation*}
    where we applied equation \ref{eq:lagrangedecaysum} to the cardinal function $\chi^{(\mathcal{I}_{i}^{(\ell)})}$ constructed on the set of points $\mathcal{I}_{i}^{(\ell)}$.
    Grouping the constants as $c_{\ref{eq:localcardinal-inf}} := c_{\ref{eq:lagrangedecayA}} \nu^{-d} N(0)c_{q}^{d}(c_{\ref{eq:lagrangedecaysum}}+1)$ leads to the desired result.
\end{proof}

Note here, that $\varrho$ is a user defined constant. $\varrho$ can be made large at the expense of enforcing the Lagrange conditions at a larger number of points. To be precise, we have 
$ \#B^{(\ell)}(\vec{x},\varrho h_{\ell} |\log(h_{\ell})|) \le (1+\varrho c_{q} |\log{h_{\ell}|})^{d}$.
We refer to $\check{\chi}$ as local version of the Lagrange functions. 

\section{Adaptivity}
\label{sec:adaptivity}
In this section we introduce the concept of adaptivity to the multiscale framework.
The overall idea is to ease the computational cost by considering at each level/ scale a subset of $X_{\ell}$, in particularly a subset were the residual is still large, while "cutting" all points where the error is already low, following the idea from \cite{LeGia:Wendland:2014}[Algorithm 2] in term of coefficients. To end this we propose the following alternative general algorithm.

\begin{algorithm}[ht]
    \caption{Adaptive multiscale approximation \label{alg:adaptms}}
    \begin{algorithmic}[1]
    
    \State{Input: Number of levels $L$, sets of points $X_{1}, \dots, X_{L}$, rhs $f$, thresholding $ \varepsilon$}
    \State{Output: Multiscale approximation $f_{L}$.}
    \vspace{0.1cm}
    \State $\bar{e}_{0} \leftarrow f$; 
    \State $\bar{f}_{0} \leftarrow 0$; 
    \For{$\ell = 1, \dots, L-1$}   
        \State $\check{X}_{\ell} \leftarrow \mathcal{T}_{\ell}(X_{\ell}, f,  \varepsilon, k, \magicC);$ \quad see  \eqref{eq:adaptpoint} for a definition;
        \State $\bar{s}_{\ell} \leftarrow \mathcal{I}_{\ell}(\bar{e}_{\ell-1})$; 
        \State $\bar{e}_{\ell} \leftarrow \bar{e}_{\ell-1} - \bar{s}_{\ell}$; 
        \State $\bar{f}_{\ell} \leftarrow \bar{f}_{\ell-1} + \bar{s}_{\ell}$; 
    \EndFor
    \end{algorithmic}
\end{algorithm}

Here $\mathcal{T}_{\ell}$ is operator responsible for the point selection. We provide a definition in \eqref{eq:adaptpoint}.

We use the notation
\begin{equation}
    \label{eq:totalpoints}
        X_{\text{total}}:=\bigcup_{\ell=1}^{L} X_{\ell}
\end{equation}
for the total number of used points.
Now we want to show that an adaptive strategy to solve the multiscale problem is possible. We assume in the following that $X_{1} \subset \dots \subset X_{\ell} \subset \dots\}$ is a nested sequence of point set that satisfies Assumption \ref{ass:pointset} with a fixed $k>1$. In the next sections we use a bar to indicate that objects belong to the adaptive strategy, e.g., $e_{\ell}$ is the error of the classic scheme at level $\ell$, while $\bar{e}_{\ell}$ is the error of the adaptive scheme at the same level. 
Clearly, it is usually more appropriate to have adaptivity where can be more profitable, i.e., there are more points involved. Thus, it is reasonable to combine the methods, assuming that the traditional algorithm is applied in the first $\ell_{0}$ levels and the adaptive from there on.  We fix $\ell_{0} := L^{\star}$, as it gives consistency (see theorem \ref{thm:consistency}), however, different choices are possible in practice. Note that $\ell_{0}=0$ provide a fully adaptive algorithm. More comments about it are present further in the article.

Moreover, in the following theoretical analysis we assume to deal just with the adaptive part, having the target function of algorithm \ref{alg:adaptms} as the error after $L^{\star}$ iteration of the classical scheme, i.e., $e_{\ell_{0}}$.

At level $\ell$, we assume to have access to the residual of the previous step, i.e.,
\begin{align*}
	\bar{e}_{\ell-1} = e_{\ell_{0}} - \sum_{j=\ell_{0}+1}^{\ell-1} \bar{s}_{j}.
\end{align*}
We define, for $\ell > \ell_{0}$, 
\begin{align*}
	X^{\prime}_{\ell}:= \left\{\vec{x}^{(\ell)} \in X_{\ell} \, : \, \forall
    \vec{x}^{(L)} \in X_{L} \cap B(\vec{x}_{\ell}, k \magicC h_{\ell} |\log(h_{\ell})|):  | \bar{e}_{\ell-1}(\vec{x}^{(L)}) | \le  \varepsilon_{\ell-1} \right\} \subset X_{\ell}.
\end{align*}
The set $X^{\prime}_{\ell}$ represents the set of points that we want to remove.
Moreover, we set 
\begin{align*}
	X^{\prime}_{\ell}\subset \bar{X}_{\ell}:= \bigcup_{\vec{x}^{(\ell)} \in X^{\prime}_{\ell}} B^{(\ell)}(\vec{x}^{(\ell)}; k \magicC h_{\ell} |\log(h_{\ell})|).
\end{align*}

In the end we have
\begin{equation}\label{eq:adaptpoint}
    \mathcal{T}_{\ell}(X_{\ell}, f, \varepsilon_{\ell},k, \magicC) := X_{\ell} \setminus \bar{X}_{\ell} = \check{X}_{\ell},
\end{equation}
where even if not explicitly written to lighten the notation, $\bar{X}_{\ell}$ depends on both $e_{\ell_{0}}$ and $\varepsilon_{\ell}$.

We then compute the approximation at level $\ell$ 
\begin{align*}
	\bar{s}_{\ell}:=\sum_{\vec{x}_{i}^{(\ell)} \in \check{X}_{\ell}} \bar{e}_{\ell-1} (\vec{x}_{i}^{(\ell)}) \chi^{(\ell)}_{i}.
\end{align*}
\begin{remark}
    \label{rk:check}
    By making $k\kappa >\varrho$, we can make sure that we can safely replace $\chi^{(\ell)}_{i}$ by $\check{\chi}^{(\ell)}_{i}$ and $\check{\chi}^{(\ell)}_{i}$ can be computed without using points from $X^{\prime}_{\ell}$.
\end{remark}

We are now in the position to assess the consistency of the adaptive scheme.
Define 
\begin{equation}
    \label{eq:pointsused}
        X_{\le \ell} := \bigcup_{j=\ell_{0}+1}^{\ell} \check{X}_{j}.
\end{equation}
    
First, we need to introduce some intermediate results.
\begin{lemma}\label{lem:ellbound}
Let Assumptions \ref{ass:kernel} and \ref{ass:pointset} hold. 
    Then, we have
    \begin{equation}
        \label{eq:ell1bound}
            \|\bar{e}_{\ell-1}\|_{\ell^{\infty}(\check{X}_{\ell})}\le
        (1+c_{\ref{eq:lagrangedecaysum}})^{\ell-\ell_{0}-1}\left\| e_{\ell_{0}} \right\|_{\ell^{\infty}(X_{\le \ell} )}.
    \end{equation}
\end{lemma}
\begin{proof}
   First, we start with
    \begin{align*}
        \|\bar{e}_{\ell-1}&\|_{\ell^{\infty}(\check{X}_{\ell})} = \|\bar{e}_{\ell-2}-\bar{s}_{\ell-1}\|_{\ell^{\infty}(\check{X}_{\ell})} \le \left\| \bar{e}_{\ell-2}\right\|_{\ell^{\infty}(\check{X}_{\ell})} + \left\|\bar{s}_{\ell-1}\right\|_{\ell^{\infty}(\check{X}_{\ell})} \\
        & \le \left\|\bar{e}_{\ell-2}\right\|_{\ell^{\infty}(\check{X}_{\ell})} + \left\| \sum_{\vec{x}_{j}^{(\ell-1)} \in \check{X}_{\ell-1}} \bar{e}_{\ell-2}(\vec{x}_{j}^{(\ell-1)}) \chi^{(\ell-1)}_{j} \right\|_{\ell^{\infty}(\check{X}_{\ell})} \\
        & \le \left\|\bar{e}_{\ell-2}\right\|_{\ell^{\infty}(\check{X}_{\ell})} + \left\| \bar{e}_{\ell-2} \right\|_{\ell^{\infty}(\check{X}_{\ell-1})} \max_{\vec{x}^{(\ell)} \in \check{X}_{\ell}}  \sum_{\vec{x}_{j}^{(\ell-1)} \in \check{X}_{\ell-1}} \left| \chi^{(\ell-1)}_{j}(\vec{x}^{(\ell)}) \right| \\
        & \le (1 + c_{\ref{eq:lagrangedecaysum}}) \left\| \bar{e}_{\ell-2} \right\|_{\ell^{\infty}(\check{X}_{\ell-1} \cup \check{X}_{\ell})} \le (1 + c_{\ref{eq:lagrangedecaysum}}) \left\| \bar{e}_{\ell-2} \right\|_{\ell^{\infty}(X_{\le \ell} )},
    \end{align*}
    where we exploited corollary \ref{cor:lagrangedecaysum}. Iterating thus yields
    \begin{equation*}
        \|\bar{e}_{\ell-1}\|_{\ell^{\infty}(\check{X}_{\ell})} \le (1+c_{\ref{eq:lagrangedecaysum}})^{\ell-\ell_{0}-1}\left\| e_{\ell_{0}} \right\|_{\ell^{\infty}(X_{\le \ell} )}.
    \end{equation*} 
\end{proof}

\begin{lemma}\label{lem:barsellbound}
Let Assumptions \ref{ass:kernel} and \ref{ass:pointset} hold.
For all $\vec{x}^{(\ell)}\in X^{\prime}_{\ell}$ we have the bound
\begin{equation}\label{eq:barsellbound}
    \|\bar{s}_{\ell}\|_{\ell^{\infty}(B^{(L)}(\vec{x}^{(\ell)}; \magicC h_{\ell} |\log (h_{\ell})|))}\le C_{\#} c_{\ref{eq:lagrangedecay}}(1+c_{\ref{eq:lagrangedecaysum}})^{\ell-\ell_{0}-1}\left\| e_{\ell_{0}} \right\|_{\ell^{\infty}(X_{\le \ell} )} h^{\magicC \theta(k-1)-d}_{\ell}.
\end{equation}
\end{lemma}
\begin{proof}
    Let $\vec{x}^{(\ell)} \in X^{\prime}_{\ell}$ be arbitrary but fixed.
    By construction, we have that 
    \begin{equation*}
        \dist(\vec{x}^{(\ell)},  \check{X}_{\ell})\ge k \magicC h_{\ell}|\log(h_{\ell})|.
    \end{equation*}
    Hence, for $\vec{x}^{(L)} \in B^{(L)}(\vec{x}^{(\ell)}; \magicC h_{\ell} |\log (h_{\ell})|)$
    \begin{equation*}
        \dist(\vec{x}^{(L)},  \check{X}_{\ell})\ge (k-1) \magicC h_{\ell}|\log(h_{\ell})|.
    \end{equation*}
    This gives, using \eqref{eq:lagrangedecay}, the bound
    \begin{align*}
    	|\bar{s}_{\ell}&(\vec{x}^{(L)})| \le \sum_{\vec{x}_{i}^{(\ell)} \in \check{X}_{\ell}} |\bar{e}_{\ell-1} (\vec{x}_{i}^{(\ell)}) \chi^{(\ell)}_{i}(\vec{x}^{(L)})| \le \sum_{\vec{x}_{i}^{(\ell)} \in \check{X}_{\ell}} c_{\ref{eq:lagrangedecay}}|\bar{e}_{\ell-1}(\vec{x}_{i}^{(\ell)})| e^{-\theta\|\vec{x}^{(L)}-\vec{x}_{i}^{(\ell)}\|_{2}/q_{\ell}}\\
        & \le \sum_{\vec{x}_{i}^{(\ell)} \in \check{X}_{\ell}} c_{\ref{eq:lagrangedecay}}|\bar{e}_{\ell-1}(\vec{x}_{i}^{(\ell)})|  h_{\ell}^{\magicC (k-1) \theta}  = c_{\ref{eq:lagrangedecay}} \|\bar{e}_{\ell-1}\|_{\ell^{1}(\check{X}_{\ell})}  h^{\magicC\theta(k-1)}_{\ell} \\
        & \le C_{\#} c_{\ref{eq:lagrangedecay}} \|\bar{e}_{\ell-1}\|_{\ell^{\infty}(\check{X}_{\ell})}  h^{\magicC \theta(k-1)-d}_{\ell}.
    \end{align*}
    The claim follows by applying Lemma \ref{lem:ellbound}.
\end{proof}

\begin{remark}
    The bound outlined in lemma \ref{lem:ellbound} is rather pessimistic, as we expect that the behaviour of $\|\bar{e}_{\ell-1}\|_{\ell^{\infty}(\check{X}_{\ell})}$ to be quite different from $\left\| \bar{e}_{\ell-2}\right\|_{\ell^{\infty}(\check{X}_{\ell})} + \left\|\bar{s}_{\ell-1}\right\|_{\ell^{\infty}(\check{X}_{\ell})}$, as the interpolant $\bar{s}_{\ell-1}$ is computed to approximate $\bar{e}_{\ell-2}$. Additionally, having $e_{\ell_{0}}$ as the error of the classic multiscale algorithm on the first levels, we have \ref{lem:classic-infbound} that allow us to write
    \begin{equation*}
        \|\bar{e}_{\ell-1}\|_{\ell^{\infty}(\check{X}_{\ell})}\le
        (1+c_{\ref{eq:lagrangedecaysum}})^{\ell-\ell_{0}-1} C_{2} (C_{1} \mu^{\tau-d/2})^{\ell_{0}} \left\| f \right\|_{H^{\tau}(\Omega)}.
    \end{equation*}
    Thus, we expect that the contribution of $(1+c_{\ref{eq:lagrangedecaysum}})^{\ell-\ell_{0}-1}$ is negligible. 
    For this reasons, the influence of $\ell$ on the bound should not be a concern. Indeed, numerical experiment shows a decay of $\|e_{\ell-1}\|_{\ell^{\infty}(\check{X}_{\ell})}$ with respect to $\ell$.
    Assuming, e.g., 
    \begin{equation*}
        \|\bar{e}_{\ell-1}\|_{\ell^{\infty}(\check{X}_{\ell})} \le h_{\ell}^{d}\left\|e_{\ell_{0}}\right\|_{\ell^{\infty}(\Omega)},
    \end{equation*}
    we can have 
    \begin{equation*}
        \|\bar{s}_{\ell}\|_{\ell^{\infty}(B^{(L)}(\vec{x}^{(\ell)}; \magicC h_{\ell} |\log (h_{\ell})|))}\le c_{\#} c_{\ref{eq:lagrangedecay}} \left\| e_{\ell_{0}} \right\|_{\ell^{\infty}(X_{\le \ell} )} h^{\magicC \theta(k-1)}_{\ell},
    \end{equation*}
    i.e., we can get the same bound with a smaller value of $\magicC$ compared with the result from the lemma.
\end{remark}
We finally introduce some structure on the threshold, and prove that such choice leads to a consistent adaptive algorithm.
\begin{definition}\label{def:eps}
    Let Assumptions \ref{ass:kernel} and \ref{ass:pointset} hold and $\magicC>1$.
    Then, for every $\varepsilon_0>0$, we can recursively define, in the context of Algorithm \ref{alg:adaptms}, for $\ell > \ell_{0}$
    \begin{equation}\label{eq:epsdef}
        \varepsilon_{\ell} := \varepsilon_{\ell-1} + \max_{\vec{x}^{(\ell)} \in X_{\ell}^{\prime}}\|s_{\ell}\|_{\ell^{\infty}(B^{(L)}(\vec{x}^{(\ell)};\magicC h_{\ell} |\log (h_{\ell})|))}.
    \end{equation}
\end{definition}

\begin{lemma}\label{lem:epsbound}
    Let Assumptions \ref{ass:kernel} and \ref{ass:pointset} hold. Let further assume that
    \begin{equation}
    \label{eq:magicCchoice}
        \magicC \ge \max \left(1, (k-1)^{-1}\theta^{-1}\left(\frac{\ln(2(1+c_{\ref{eq:lagrangedecaysum}}))}{\ln(\mu^{-1})} +d\right) \right)
    \end{equation}
    and consider $\varepsilon_{\ell}$ as defined in Definition \ref{def:eps}. Then, 
    \begin{equation}
        \varepsilon_{\ell}\le \varepsilon_{\ell_{0}}+c_{\#} c_{\ref{eq:lagrangedecay}} \left\| e_{\ell_{0}} \right\|_{\ell^{\infty}(X_{\le \ell} )}2c_{h}h_{\ell_{0}+1}^{\magicC \theta(k-1)-d}
    \end{equation}
    holds.
\end{lemma}
\begin{proof}
    Iterating over the result of Lemma \ref{lem:barsellbound} gives
    \begin{align*}
        \varepsilon_{\ell}&= \varepsilon_{\ell-1} + \max_{\vec{x}^{(\ell)} \in X_{\ell}^{\prime}}\|s_{\ell}\|_{\ell^{\infty}(B^{(L)}(\vec{x}^{(\ell)};\magicC h_{\ell} |\log (h_{\ell})|))} \\
        & \le \varepsilon_{\ell-1} + c_{\#} c_{\ref{eq:lagrangedecay}}(1+c_{\ref{eq:lagrangedecaysum}})^{\ell-\ell_{0}-1}\left\| e_{\ell_{0}} \right\|_{\ell^{\infty}(X_{\le \ell} )} h^{\magicC \theta(k-1)-d}_{\ell}\\
        &\le \varepsilon_{\ell_{0}}+c_{\#} c_{\ref{eq:lagrangedecay}}\left\| e_{\ell_{0}} \right\|_{\ell^{\infty}(X_{\le \ell} )}\sum_{j=\ell_{0}+1}^{\ell}(1+c_{\ref{eq:lagrangedecaysum}})^{j-\ell_{0}-1} h^{\magicC \theta(k-1)-d}_{j} \\
        &\le \varepsilon_{\ell_{0}}+c_{\#} c_{\ref{eq:lagrangedecay}}\left\| e_{\ell_{0}} \right\|_{\ell^{\infty}(X_{\le \ell} )} (1+c_{\ref{eq:lagrangedecaysum}})^{-1} \sum_{j=1}^{\ell}(1+c_{\ref{eq:lagrangedecaysum}})^{j} h^{\magicC \theta(k-1)-d}_{j+\ell_{0}}.
        \end{align*}
        Using Assumption \ref{ass:pointset}, we obtain by geometric series due to \eqref{eq:magicCchoice}
        \begin{equation*}
            \sum_{j=1}^{\ell}(1+c_{\ref{eq:lagrangedecaysum}})^{j} h^{\magicC \theta(k-1)-d}_{j+\ell_{0}} \le h^{\magicC \theta(k-1)-d}_{\ell_{0}} 
            \frac{\mu^{\magicC \theta(k-1)-d}(1+c_{\ref{eq:lagrangedecaysum}})}{1-\mu^{\magicC \theta(k-1)-d}(1+c_{\ref{eq:lagrangedecaysum}})}\le
            2c_{h}h_{\ell_{0}+1}^{\magicC \theta(k-1)-d}(1+c_{\ref{eq:lagrangedecaysum}}).
        \end{equation*}
    This proves the claim.
\end{proof}

This result highlights how, among the many contributes on the recursive definition, the first terms, as expected have a major influence on the bound of further levels, due to $\mu<1$. Indeed, as outlined at the beginning of the paragraph, an effective strategy is of combining the classic scheme for the first levels and then switch to the adaptive scheme, where the computational benefits are more meaningful. This implies $h_{\ell_{0}+1} \ll 1$, thus $\varepsilon_{\ell} \approx \varepsilon_{\ell_{0}}$. Alternatively, a proper choice of $h_{1}$ does also the trick. However, small values of $h_{1}$ are associated with large sets of points, and starting with already a rich set of points work against the potential of multiscale framework. Moreover, is worth to point out that the bound is $\ell$-independent. Indeed, bounding $\left\| e_{\ell_{0}} \right\|_{\ell^{\infty}(X_{\le \ell} )}$ can be achieved by the highest value of the rhs of the multilevel interpolation problem, and $h_{\ell_{0}+1}$ can be arbitrary small with proper choices of the parameters $\magicC$ and $k$ even with respect to fixed $\ell_{0} \in \N$.

We are now in position to state the consistency result we are interested in.

\begin{theorem}\label{thm:consistency}
    Let Assumptions \ref{ass:kernel} and \ref{ass:pointset} hold. Let $\magicC>1$ and $\varepsilon_{\ell}$ as defined in \ref{def:eps}.
    Then, algorithm \ref{alg:adaptms} produces sets such that
    \begin{equation}\label{eq:consistency}
         X^{\prime}_{\ell} \subset X^{\prime}_{\ell+1}, \qquad L^{\star} \le \ell \le L-1
    \end{equation}
    holds.
\end{theorem}
\begin{remark}
    Relation \eqref{eq:consistency} means that every point adaptively removed at a given step would be removed also on each subsequent iteration, hence it is safe to discard this point.
    Please note that the removal criterion \eqref{eq:adaptpoint} depends on the values of $\varepsilon_{\ell}$. Also, note that here $\ell$ takes into account also eventual classic steps, as
\end{remark}
\begin{proof}
    Given $\vec{x}^{(\ell)} \in X^{\prime}_{\ell}$, we aim to show that
    \begin{equation*}
    	\|\bar{e}_{\ell}\|_{\ell^{\infty}(B^{(L)}(\vec{x}^{(\ell)}; k \magicC h_{\ell+1} |\log (h_{\ell+1})|))} \le \varepsilon_{\ell} 
    \end{equation*}
    holds, to ensure that $\vec{x}^{(\ell)} \in X^{\prime}_{\ell+1}$. 
    Since $\ell \ge L^{\star}$ we have, due to Lemma \ref{lem:Lstar} ,
    \begin{equation*}
        X_{\ell+1} \cap B(\vec{x}^{(\ell)}; k \magicC h_{\ell+1} |\log (h_{\ell+1})| )\subset X_{\ell+1} \cap B(\vec{x}^{(\ell)}; \magicC h_{\ell} |\log (h_{\ell})|)
    \end{equation*}
    for each $\vec{x}^{(\ell)} \in X^{\prime}_{\ell} \subset X_{\ell} \subset X_{\ell+1}$.
    Indeed, we have 
    \begin{equation*}
        \|\bar{e}_{\ell}\|_{\ell^{\infty}(B^{(L)}(\vec{x}^{(\ell)}; k \magicC h_{\ell+1} |\log (h_{\ell+1})|))} \le \|\bar{e}_{\ell}\|_{\ell^{\infty}(B^{(L)}(\vec{x}^{(\ell)}; \magicC h_{\ell} |\log (h_{\ell})|))}.
    \end{equation*}
    Thus, for
    \begin{equation*}
        \vec{x}^{(L)} \in B^{(L)}(\vec{x}^{(\ell)}; \magicC h_{\ell} |\log (h_{\ell})|) \subset  B^{(L)}(\vec{x}^{(\ell)}; k \magicC h_{\ell+1} |\log (h_{\ell+1})|), 
    \end{equation*}
    Lemma \ref{lem:barsellbound} allows us to show
    \begin{align*}
    	|\bar{e}_{\ell}(\vec{x}^{(L)})| & \le |\bar{e}_{\ell-1}(\vec{x}^{(L)})|  + |\bar{s}_{\ell}(\vec{x}^{(L)})|  \\
        & \le \varepsilon_{\ell-1} + \max_{\vec{x}^{(\ell)} \in X_{\ell}^{\prime}}\|s_{\ell}\|_{\ell^{\infty}(B^{(L)}(\vec{x}^{(\ell)};\magicC h_{\ell} |\log (h_{\ell})|))},
    \end{align*}
    concluding that $|\bar{e}_{\ell}(\vec{x}^{(L)})| \le \varepsilon_{\ell}$ for $\vec{x}^{(L)} \in B^{(L)}(\vec{x}^{(\ell)}; k \magicC h_{\ell+1} |\log (h_{\ell+1})|))$. In particular, for $\vec{x}^{(\ell)} \in X_{\ell+1}$, we have $\vec{x}^{(\ell)} \in X^{\prime}_{\ell+1}$.
\end{proof}

We now analyse the special setting on which the target function has compact support, where the adaptivity can express its full potential. Indeed, a classic multilevel approach with $L$ steps would still employ points far from the support to compute the approximant, while with the adaptive version we hope to deal with just the point "close" to the support.

\begin{theorem}\label{thm:compactsupp}
    Let $r \in \R$ be such that exists $\vec{x} \in \Omega$ for which $B(r) := B_{r}(\vec{x}) \supset \overline{\{\vec{y} \in \Omega : e_{\ell_{0}}(\upsilon) \neq 0\}}$, i.e. the support of $e_{\ell_{0}}$ is contained into a ball of radius $r$ with fixed $\vec{x}$. For $L \in \N$, let also $X_{1}, \dots, X_{L}$ be a nested sequence of sets of points satisfying Assumption \ref{ass:pointset} and $\{\Phi_{\ell}\}_{1\le \ell \le L}$ satisfying Assumption \ref{ass:kernel}. Then, for every $\ell > \ell_{0}$, we have
    \begin{equation}
        X^{\prime}_{\ell} \supseteq X_{\ell} \setminus B^{(\ell)}(r+k \magicC\rho)
    \end{equation}
    with $\rho = 2 h_{\ell_{0}}\frac{\mu\log(1/\mu)+(1-\mu)\log(1/h_{\ell_{0}})}{(1-\mu)^{2}}$, for every point refinement $\check{X}_{\ell}$ computed with Algorithm \ref{alg:adaptms}.
\end{theorem}

\begin{proof}
    In the following we define $\rho_{\ell} := h_{\ell} |\log(h_{\ell})|$ and assume $\vec{x} \in \Omega$ is fixed.
    On every adaptive step we compute, for $j > \ell_{0}$,
    \begin{align*}
    	X^{\prime}_{\ell}:= \left\{\vec{x}^{(\ell)} \in X_{\ell} \, : \, \forall
        \vec{x}^{(L)} \in X_{L},  | \bar{e}_{\ell-1}(\vec{x}^{(L)}) | \le  \varepsilon_{\ell-1},  \|\vec{x}^{(L)}-\vec{x}^{(\ell)}\|_2 \le k \magicC\rho_{\ell}\right\},
    \end{align*}
    \begin{align*}
    	X^{\prime}_{\ell}\subset \bar{X}_{\ell}:= \bigcup_{\vec{x}^{(\ell)} \in X^{\prime}_{\ell}} \left( X_{\ell} \cap B(\vec{x}^{(\ell)}; k \magicC\rho_{\ell})\right),
    \end{align*}
    \begin{equation*}
        \bar{s}_{j} =\sum_{\vec{x}^{(\ell)} \in X_{j} \setminus \bar{X}_{j}} \bar{e}_{j-1} (\vec{x}^{(\ell)}_{i}) \chi^{(j)}_{i}.
    \end{equation*}

    Precisely, for $j=\ell_{0}+1$, we have that $\bar{e}_{j-1} = e_{\ell_{0}}$ vanish outside the support of $e_{\ell_{0}}$, in particular outside of $B(r)$. 
    Thus, we have that
    \begin{equation*}
        X^{\prime}_{\ell_{0}+1} \supseteq X_{\ell_{0}+1} \setminus B^{(\ell_{0}+1)}(r+k \magicC\rho_{\ell_{0}+1}), 
    \end{equation*}
    for every choice of $\varepsilon_{\ell_{0}} >0$.
    From here we iterate the process, and prove by induction the statement.
    $j=\ell_{0}+1$ is already proven; now assuming $j=\ell$ holds, i.e.,
    \begin{equation*}
        X^{\prime}_{\ell} \supseteq X_{\ell} \setminus B^{(\ell)}(r+k \magicC\rho_{\ell}+2k \magicC\sum\nolimits_{i=\ell_{0}+1}^{\ell-1}\rho_{i})),
    \end{equation*}
    we prove $j = \ell+1$. 
    First, for $X_{\ell+1} \ni \vec{y}^{(\ell+1)} \notin B^{(\ell+1)}(r+2k \magicC\sum_{i=\ell_{0}+1}^{\ell}\rho_{i})$ we have
    \begin{equation*}
        |\bar{e}_{\ell}(\vec{y}^{(\ell+1)})| = |\bar{e}_{\ell-1}(\vec{y}^{(\ell+1)}) - \bar{s}_{\ell}(\vec{y}^{(\ell+1)})|.
    \end{equation*}
    Moreover, we have that
    \begin{equation*}
        \|\vec{y}^{(\ell+1)}-\vec{x}^{(\ell)}\|_{2} > k \magicC\rho_{\ell} > \magicC \rho \quad \forall \vec{x}^{(\ell)} \in B^{(\ell)}(r+k \magicC\rho_{\ell}+2\sum\nolimits_{i=\ell_{0}+1}^{\ell-1}k \magicC\rho_{i}).
    \end{equation*}
    Thus, $\exists \, \vec{x}_{\star}^{(\ell)} \in X^{(\ell)} \cap B_{\magicC\rho_{\ell}}(\vec{y}^{(\ell+1)})$ such that
    \begin{equation*}
        \vec{x}_{\star}^{(\ell)} \notin B^{(\ell)}(r+k \magicC\rho_{\ell}+2k \magicC\sum\nolimits_{i=\ell_{0}+1}^{\ell-1}\rho_{i}),
    \end{equation*}
    which implies, by induction hypothesis, that $\vec{x}_{\star}^{(\ell)} \in X^{\prime}_{\ell}$.
    Additionally, we have $\vec{y}^{(\ell+1)} \in X_{\ell+1} \subset X_{L}$, thus, by construction,
    \begin{equation*}
        | \bar{e}_{\ell-1}(\vec{y}^{(\ell+1)}) | \le  \varepsilon_{\ell-1}, \quad \text{ and } \quad
        |\bar{s}_{\ell}(\vec{y}^{(\ell+1)})| \le \|\bar{s}_{\ell}\|_{\ell^{\infty}(B^{(L)}(\vec{x}_{\star}^{(\ell)};\magicC\rho_{\ell}))}.
    \end{equation*}
    Therefore, we have
    \begin{equation*}
        |\bar{e}_{\ell}(\vec{y}^{(\ell+1)})| \le  |\bar{e}_{\ell-1}(\vec{y}^{(\ell+1)})| + |\bar{s}_{\ell}(\vec{y}^{(\ell+1)})| \le \varepsilon_{\ell}
    \end{equation*}
    for every $\vec{y}^{(\ell+1)} \notin B^{(\ell+1)}(r+2k\magicC\sum_{i=\ell_{0}+1}^{\ell}\rho_{i})$.

    Additionally, given $X_{\ell+1} \ni \vec{x}^{(\ell+1)} \notin B^{(\ell+1)}(r+k\magicC\rho_{\ell+1}+2k\magicC\sum_{i=\ell_{0}+1}^{\ell}\rho_{i})$, we have
    \begin{equation*}
        X_{L} \cap B_{k\magicC\rho_{\ell+1}}(\vec{x}^{(\ell+1)}) \subset X_{L} \setminus B^{(L)}(r+2k\magicC\sum_{i=\ell_{0}+1}^{\ell}\rho_{i}).
    \end{equation*}
    Thus, $\vec{x}^{(\ell+1)} \in X^{\prime}_{\ell+1}$, which terminates the induction proof.
    Lastly, from Assumption \ref{ass:pointset}, it follows that
    \begin{align*}
        \sum_{i=\ell_{0}+1}^{\infty} h_{i}|\log(h_{i})| &\le \sum_{i=1}^{\infty}h_{\ell_{0}}\mu^{i} (i\log(1/\mu)+\log(1/h_{\ell_{0}})) \\
        &\le h_{\ell_{0}}\log(1/\mu)\sum_{i=1}^{\infty}i\mu^{i}+h_{\ell_{0}}\log(1/h_{\ell_{0}})\sum_{i=1}^{\infty}\mu^{i} \\
        &\le h_{\ell_{0}}\log(1/\mu)\frac{\mu}{(1-\mu)^{2}}+h_{\ell_{0}}
        \log(1/h_{\ell_{0}})\frac{1}{1-\mu} \\
        &\le h_{\ell_{0}}\frac{\mu\log(1/\mu)+(1-\mu)\log(1/h_{\ell_{0}})}{(1-\mu)^{2}}.
    \end{align*}
\end{proof}

\begin{corollary}
    Let $r \in \R$ be such that exists $\vec{x} \in \Omega$ for which $B(r) := B_{r}(\vec{x}) \supset \overline{\{\vec{y} \in \Omega : f(\upsilon) \neq 0\}}$, i.e. the support of $f$ is contained into a ball of radius $r$ with fixed $\vec{x}$. For $L \in \N$, let also $X_{1}, \dots, X_{L}$ be a nested sequence of sets of points satisfying Assumption \ref{ass:pointset} and $\{\Phi_{\ell}\}_{1\le \ell \le L}$ satisfying Assumption \ref{ass:kernel}. Then, for every $\ell > \ell_{0}$, we have
    \begin{equation}
        \check{X}_{\ell} \subseteq B^{(\ell)}(r + k \magicC\rho)
    \end{equation}
    with $\rho = 2 h_{\ell_{0}}\frac{\mu\log(1/\mu)+(1-\mu)\log(1/h_{\ell_{0}})}{(1-\mu)^{2}}$, for every point refinement $\check{X}_{\ell}$ computed with Algorithm \ref{alg:adaptms}.
\end{corollary}

Proposition \ref{thm:compactsupp} highlights the algorithm behavior in the circumstances of a localized error. Compactly supported functions requires an effective number of points that is proportional to the support rather than to the domain. Additionally, it can be seen how the choice of $\mu$ influences such limits. Indeed, $\mu$ tunes the radius of the ball that contains all the points that are actually used in the approximation. In particular, we have
\begin{equation*}
    \rho \xrightarrow{\mu \rightarrow 0} 0.   
\end{equation*}

Thus, for $\mu \rightarrow 0$, we have

\begin{equation*}
    X^{\prime}_{\ell} \supset X_{\ell} \setminus B^{(\ell)}(r), \quad \check{X}_{\ell} \subset B^{\ell}(r).
\end{equation*}

Another simple but effective result to reduce the computational time is the following.
\begin{lemma}\label{lem:checkset}
    For $L \in \N$, let also $X_{1}, \dots, X_{L}$ be a nested sequence of sets of points satisfying Assumption \ref{ass:pointset}.
    Then, using the notation of Algorithm \ref{alg:adaptms}, given $\vec{x}^{(\ell)} \in X_{\ell}$ such that $|\bar{e}_{\ell-1}(\vec{x}^{(\ell)})| > \varepsilon_{\ell-1}$, we have $\vec{x}^{(\ell)} \in \check{X}_{\ell}$.
\end{lemma}
\begin{proof}
    Assume instead that $\vec{x}^{(\ell)} \notin \check{X}_{\ell} = X_{\ell} \setminus \bar{X}_{\ell}$. Thus, $\vec{x}^{(\ell)} \in \bar{X}_{\ell}$, which also implies $\exists \, \vec{y}^{(\ell)} \in X^{\prime}_{\ell}$ such that $\|\vec{x}^{(\ell)}-\vec{y}^{(\ell)}\|_{2} \le k \magicC h_{\ell}|\log(h_{\ell})|$. However, by construction, we have $| \bar{e}_{\ell-1}(\vec{x}^{(L)}) | \le  \varepsilon_{\ell-1}$ for every $\vec{x}^{(L)} \in X_{L}$ such $\|\vec{x}^{(L)}-\vec{y}^{(\ell)}\|_2 \le k \magicC h_{\ell}|\log(h_{\ell})|$, which also implies $|\bar{e}_{\ell-1}(\vec{x}^{(\ell)})| \le \varepsilon_{\ell-1}$. Therefore, the claim.
\end{proof}

We finally, address the convergence of the adaptive scheme.
Prior to stating the convergence result for the adaptive scheme, we collect an intermediate result on the error between the adaptive and the classic multiscale scheme.

\begin{lemma}\label{lem:adapt-classic}
    Under Assumption \ref{ass:kernel} and \ref{ass:pointset}, it exists a positive constant $c_{\ref{eq:adapt-classic}}$ such that
    \begin{equation}\label{eq:adapt-classic}
        \|I_{X_{\ell}, \Phi_{\ell}} E\bar{e}_{\ell-1}-E\bar{s}_{\ell}\|_{\Phi_{\ell+1}} \le c_{\ref{eq:adapt-classic}} \varepsilon_{\ell-1},
    \end{equation}
    i.e., the difference between the adaptive and the classic interpolation applied to $e_{\ell-1}$ in the $\Phi_{\ell+1}$-norm is bounded by a constant times $\varepsilon_{\ell-1}$.
\end{lemma}
\begin{proof}
    The main idea of the proof comes from \cite{LeGia:Wendland:2014}. 
    Recalling the definition of $X^{\prime}_{\ell}$ and $\bar{X}_{\ell}$, we have
    \begin{align*}
        \|I_{X_{\ell}, \Phi_{\ell}} E\bar{e}_{\ell-1}-E\bar{s}_{\ell}\|_{\Phi_{\ell+1}} & = \left\|\sum\nolimits_{\vec{x}_{i}^{(\ell)} \in X_{\ell}} \bar{e}_{\ell-1}(\vec{x}_{i}^{(\ell)})\chi^{(\ell)}_{i} -\sum\nolimits_{\vec{x}_{i}^{(\ell)} \in X_{\ell} \setminus \bar{X}_{\ell}} \bar{e}_{\ell-1}(\vec{x}_{i}^{(\ell)})\chi^{(\ell)}_{i} \right\|_{\Phi_{\ell+1}}\\
        &= \left\|\sum\nolimits_{\vec{x}_{i}^{(\ell)} \in \bar{X}_{\ell}} \bar{e}_{\ell-1}(\vec{x}_{i}^{(\ell)})\chi^{(\ell)}_{i}\right\|_{\Phi_{\ell+1}} \le \varepsilon_{\ell-1} \left\|\sum\nolimits_{\vec{x}_{i}^{(\ell)} \in \bar{X}_{\ell}}\chi^{(\ell)}_{i}\right\|_{\Phi_{\ell+1}}\\
        & \le \left(\frac{C_{\Phi}}{c_{\Phi}}\right) \varepsilon_{\ell-1} \left\|\sum\nolimits_{\vec{x}_{i}^{(\ell)} \in \bar{X}_{\ell}}\chi^{(\ell)}_{i}\right\|_{\Phi_{\ell}},
    \end{align*}
    where the last step is motivated by \cite[Lemma 2.1]{LeGia:Wendland:2014}
    Finally, we obtain
    \begin{align*}
         \left\|\sum\nolimits_{\vec{x}_{i}^{(\ell)} \in \bar{X}_{\ell}}\chi^{(\ell)}_{i}\right\|_{\Phi_{\ell}}^{2} & = \left\|\sum\nolimits_{i,j=1}^{N(\ell)} \left(A^{-1}\right)_{i,j}\Phi_{\ell}(\vec{x}_{j}^{(\ell)}, \cdot)\right\|_{\Phi_{\ell}}^{2} \\
         & = \sum\nolimits_{i,j=1}^{N(\ell)} \sum\nolimits_{h,k=1}^{N(\ell)} \left(A^{-1}\right)_{i,j} \left(A^{-1}\right)_{h,k} \Phi_{\ell}(\vec{x}_{j}^{(\ell)}, \vec{x}_{k}^{(\ell)}) \\
         & = \sum\nolimits_{i,j=1}^{N(\ell)} \sum\nolimits_{h=1}^{N(\ell)} \left(A^{-1}\right)_{i,j} \chi^{(\ell)}_{h}(\vec{x}_{j}^{(\ell)}) \\
         & = \sum\nolimits_{i,j=1}^{N(\ell)} \left(A^{-1}\right)_{i,j},
    \end{align*}
    and, recalling Corollary \ref{cor:lagrangedecaysum}, we get
    \begin{equation*}
        \left\|\sum\nolimits_{\vec{x}_{i}^{(\ell)} \in \bar{X}_{\ell}}\chi^{(\ell)}_{i}\right\|_{\Phi_{\ell}}^{2} \le c_{\ref{eq:lagrangedecayAsum}}.
    \end{equation*}
    Putting everything together and grouping some constant under $c_{\ref{eq:adapt-classic}}^{2} = c_{\ref{eq:lagrangedecayAsum}}\left(\frac{C_{\Phi}}{c_{\Phi}}\right)^{2}$, concludes the proof.
\end{proof}

Finally, we can state the convergence of the adaptive multiscale approximation scheme:
\begin{theorem}\label{thm:adapt_errorbound}
    Under Assumption \ref{ass:kernel} and \ref{ass:pointset}, exist positive constants $C_{1}, c_{\ref{eq:adapt-classic}}$ and a sequence $\varepsilon_{\ell} > 0$, such that
    \begin{equation}\label{eq:adapterrorboundPhi}
        \|E\bar{e}_{\ell}\|_{\Phi_{\ell+1}} \le \alpha \|E\bar{e}_{\ell-1}\|_{\Phi_{\ell}} + c_{\ref{eq:adapt-classic}}\varepsilon_{\ell-1},
    \end{equation}    
    where $\alpha = C_{1}\mu^{\tau}$.
    Additionally, after $\ell \in \N$ steps, the error of the adaptive multiscale approximation can be bounded by
    \begin{equation}\label{eq:adapterrorbound}
        \|e_{\ell_{0}} - \bar{f}_{\ell}\|_{L^{2}(\Omega)} \le  c_{\ref{eq:adapterrorboundPhi}} \alpha^{\ell-\ell_{0}} \|e_{\ell_{0}}\|_{\Phi_{\ell_{0}+1}} + c_{\ref{eq:adapterrorbound}}\varepsilon_{\ell-1} \frac{1-\alpha^{\ell-\ell_{0}}}{1-\alpha}.
    \end{equation}
\end{theorem}
\begin{proof}
    Following \cite{LeGia:Wendland:2014} we can split the error 
    \begin{align*}
        \|E\bar{e}_{\ell}\|_{\Phi_{\ell+1}} & = \|E\bar{e}_{\ell-1}-E\bar{s}_{\ell}\|_{\Phi_{\ell+1}} \\
        & \le \|E\bar{e}_{\ell-1}-I_{X_{\ell}, \Phi_{\ell}} E\bar{e}_{\ell-1}\|_{\Phi_{\ell+1}} + \|I_{X_{\ell}, \Phi_{\ell}} E\bar{e}_{\ell-1}-E\bar{s}_{\ell}\|_{\Phi_{\ell+1}}.
    \end{align*}    
    We are now left to deal two pieces, where the latter is bound by Lemma \ref{lem:adapt-classic}. 
    On the other hand, since  $E$ is a linear operator, we have
    \begin{equation*}
         Ee_{j} = Ee_{j-1} - Es_{j} = Ee_{j-1} - I_{X_{j}, \Phi_{j}} Ee_{j-1}.
    \end{equation*}
    Additionally, $\bar{e}_{\ell-1} \in H^{\tau}(\Omega)$, thus we can employ the first result of theorem \ref{thm:errorbound} in our estimate and bound the former 
    \begin{equation*}
        \|E\bar{e}_{\ell-1}-I_{X_{\ell}, \Phi_{\ell}} E\bar{e}_{\ell-1}\|_{\Phi_{\ell+1}} \le \alpha \|E\bar{e}_{\ell-1}\|_{\Phi_{\ell}},
    \end{equation*}
    providing the desired result. Indeed,
    \begin{equation*}
        \|E\bar{e}_{\ell}\|_{\Phi_{\ell+1}} \le \alpha \|E\bar{e}_{\ell-1}\|_{\Phi_{\ell}} +  c_{\ref{eq:adapt-classic}}\varepsilon_{\ell-1},
    \end{equation*}
    holds, with $\alpha = C_{1}\mu^{\tau}$.
    Moreover, this implies
    \begin{align*}
        \|e_{\ell_{0}} - \bar{f}_{\ell}\|_{L^{2}(\Omega)} & \le \|\bar{e}_{\ell-1} -\bar{s}_{\ell}\|_{L^{2}(\Omega)} \\
        & \le \|\bar{e}_{\ell-1} - I_{X_{\ell}, \Phi_{\ell}} \bar{e}_{\ell-1}\|_{L^{2}(\Omega)} + \|I_{X_{\ell}, \Phi_{\ell}} \bar{e}_{\ell-1} -\bar{s}_{\ell}\|_{L^{2}(\Omega)}.
    \end{align*}
    We can bound the first term of the right hand side in the standard way, i.e. using a zeros-type inequality and the norm equivalence of $H^{\tau}$ with $\mathcal{N}_{\Phi_{\ell}}$ (see e.g. \cite{Wendland:2010}),
    \begin{align*}
        \|\bar{e}_{\ell-1} -I_{X_{\ell}, \Phi_{\ell}} \bar{e}_{\ell-1}\|_{L^{2}(\Omega)} & \le c_{b} h_{\ell}^{\tau} \|\bar{e}_{\ell-1} -I_{X_{\ell}, \Phi_{\ell}} \bar{e}_{\ell-1}\|_{H^{\tau}(\Omega)} \\
        &\le c_{b} h_{\ell}^{\tau} \|E\bar{e}_{\ell-1} -I_{X_{\ell}, \Phi_{\ell}} E\bar{e}_{\ell-1}\|_{H^{\tau}(\R^{d})} \\
        & \le c_{b} C_{\Phi}^{1/2} h_{\ell}^{\tau}/\delta_{\ell+1}^{\tau} \|E\bar{e}_{\ell-1} -I_{X_{\ell}, \Phi_{\ell}} E\bar{e}_{\ell-1}\|_{\Phi_{\ell+1}} \\
        & \le c_{b} C_{\Phi}^{1/2} (c_{h}\gamma)^{-\tau} \alpha \|E\bar{e}_{\ell-1}\|_{\Phi_{\ell}} = c_{\ref{eq:adapterrorboundPhi}} \alpha \|E\bar{e}_{\ell-1}\|_{\Phi_{\ell}}.
    \end{align*}
    On the other hand, we have
    \begin{equation*}
        \|I_{X_{\ell}, \Phi_{\ell}}\bar{e}_{\ell-1} -\bar{s}_{\ell}\|_{L^{2}(\Omega)} = \left\|\sum\nolimits_{\vec{x} \in \bar{X}_{\ell}} \bar{e}_{\ell-1}(\vec{x})\chi^{(\ell)}_{\vec{x}}\right\|_{L^{2}(\Omega)} \le \sum_{\vec{x}_{i}^{(\ell)} \in \bar{X}_{\ell}} |\bar{e}_{\ell-1}(\vec{x}_{i}^{(\ell)})| \|\chi^{(\ell)}_{i}\|_{L^{2}(\Omega)}.
    \end{equation*}
    To bound $\|\chi^{(\ell)}_{i}\|_{L^{2}(\Omega)}$ we can exploit the counting argument of Narcowich and Ward introduced in Corollary \ref{cor:lagrangedecaysum}. Indeed,
    \begin{align*}
        \|\chi^{(\ell)}_{i}&\|_{L^{2}(\Omega)}^{2} = \int_{\Omega} |\chi^{(\ell)}_{i}(\vec{x}_{i}^{(\ell)})|^{2} d\vec{x} \le \int_{\Omega} c_{\ref{eq:lagrangedecay}}^{2} e^{-2\theta\|\vec{x}_{i}^{(\ell)}-\vec{x}\|_{2}/q_{\ell}} d\vec{x} \\
         & \le \sum_{m=0}^{\infty} \int_{\mathcal{A}_{\vec{x}}^{(\ell)}(m)} c_{\ref{eq:lagrangedecay}}^{2} e^{-2\theta\|\vec{x}_{i}^{(\ell)}-\vec{x}\|/q_{\ell}} d\vec{x} \le c_{\ref{eq:lagrangedecay}}\sum_{m=0}^{\infty} \text{Vol}(\mathcal{A}_{\vec{x}}^{(\ell)}(m)) e^{-2\theta m} \\
         & \le c_{\ref{eq:lagrangedecay}} q_{\ell}^{d}\sigma_{d-1}\sum_{m=0}^{\infty} ((m+1)^{d}-m^{d}) e^{-2\theta m}.
    \end{align*}
    From here is sufficient to notice that
    \begin{equation*}
        \sum_{m=0}^{\infty} ((m+1)^{d}-m^{d}) e^{-2\theta m} \le \sum_{m=0}^{\infty} d (m+1)^{d-1} e^{-2\theta m} \le c_{\Sigma}
    \end{equation*}
    converge to a finite constant, which depends on $d$, $\theta$. Finally, recalling that $|\bar{e}_{\ell-1}|$ is bounded by $\varepsilon
    _{\ell-1}$ by construction, we have
    \begin{equation*}
         \|I_{X_{\ell}, \Phi_{\ell}}\bar{e}_{\ell-1} -\bar{s}_{\ell}\|_{L^{2}(\Omega)} \le c_{\ref{eq:lagrangedecay}} c_{\Sigma}^{1/2} \sigma_{d-1}C_{\#} \varepsilon_{\ell-1} = c_{\star}\varepsilon_{\ell-1},
    \end{equation*}
    where we bounded $\#(\bar{X}_{\ell})$ with $C_{\#}h_{\ell}^{d}$.
    Finally, we have
    \begin{equation*}
         \|e_{\ell_{0}} - \bar{f}_{\ell}\|_{L^{2}(\Omega)} \le c_{\ref{eq:adapterrorboundPhi}} \alpha \|\bar{e}_{\ell-1}\|_{\Phi_{\ell}} +c_{\star}\varepsilon_{\ell-1}.
    \end{equation*}
    Lastly, we can apply equation \eqref{eq:adapterrorboundPhi} $\ell-\ell_{0}-1$ times and get
    \begin{align*}
        \|e_{\ell_{0}} - \bar{f}_{\ell}\|_{L^{2}(\Omega)} &\le c_{\ref{eq:adapterrorboundPhi}} \alpha^{\ell-\ell_{0}} \|e_{\ell_{0}}\|_{\Phi_{\ell_{0}+1}} + c_{\ref{eq:adapterrorboundPhi}} c_{\ref{eq:adapt-classic}} \sum_{i=1}^{\ell-\ell_{0}-1}\alpha^{i}\varepsilon_{\ell-1-i} +c_{\star}\varepsilon_{\ell-1} \\
        &\le c_{\ref{eq:adapterrorboundPhi}} \alpha^{\ell-\ell_{0}} \|e_{\ell_{0}}\|_{\Phi_{\ell_{0}+1}} + \max(c_{\ref{eq:adapterrorboundPhi}} c_{\ref{eq:adapt-classic}}, c_{\star}) \sum_{i=0}^{\ell-\ell_{0}-1}\alpha^{i}\varepsilon_{\ell-1-i} \\
        & \le c_{\ref{eq:adapterrorboundPhi}} \alpha^{\ell-\ell_{0}} \|e_{\ell_{0}}\|_{\Phi_{\ell_{0}+1}} + \max(c_{\ref{eq:adapterrorboundPhi}} c_{\ref{eq:adapt-classic}}, c_{\star}) \varepsilon_{\ell-1} \frac{1-\alpha^{\ell-\ell_{0}}}{1-\alpha}\\
        & = c_{\ref{eq:adapterrorboundPhi}} \alpha^{\ell-\ell_{0}} \|e_{\ell_{0}}\|_{\Phi_{\ell_{0}+1}} + c_{\ref{eq:adapterrorbound}}\varepsilon_{\ell-1} \frac{1-\alpha^{\ell-\ell_{0}}}{1-\alpha},
    \end{align*}
    which concludes the proof.
\end{proof}

Theorem \ref{thm:adapt_errorbound} finally highlights the effect of adaptivity in the scheme. Indeed, recalling \ref{lem:epsbound}, we have, for $C_{1}\mu^{\tau} < 1$,
\begin{equation}\label{eq:adapt_bound_implications}
    c_{\ref{eq:adapterrorboundPhi}} \alpha^{\ell-\ell_{0}} \|f\|_{\Phi_{\ell_{0}+1}} + c_{\ref{eq:adapterrorbound}}\varepsilon_{\ell-1} \frac{1-\alpha^{\ell-\ell_{0}}}{1-\alpha} \xrightarrow{\quad \ell \rightarrow \infty \quad} \frac{c_{\ref{eq:adapterrorbound}}}{1-\alpha}\varepsilon_{\ell-1},
\end{equation}
which is bounded by
\begin{equation*}
    \frac{c_{\ref{eq:adapterrorbound}}}{1-\alpha}\varepsilon_{\ell-1} \le \frac{c_{\ref{eq:adapterrorbound}}}{1-\alpha}(\varepsilon_{0}+c_{\ref{eq:adapt_bound_implications}}\left\| e_{\ell_{0}} \right\|_{\ell^{\infty}(X_{\le \ell} )}h_{\ell_{0}+1}^{\magicC \theta(k-1)-d}).
\end{equation*}

Indeed, the bound ultimately behaves as $\frac{c_{\ref{eq:adapterrorbound}}}{1-\alpha}\varepsilon_{0}$ for $h_{\ell_{0}+1}$ small enough. Therefore, the choice of the parameter $\varepsilon_{0}$ affects the overall accuracy of the scheme. 

Although the algorithm \ref{alg:adaptms} update the solution using a selection of the points available at a given level, the computation of cardinal functions still require the full set. Indeed, to compute $\chi_{i}^{(\ell)}$ with $\vec{x}_{i}^{(\ell)} \in X_{\ell} \setminus X^{\prime}_{\ell}$ still require to solve a $N(\ell) \times N(\ell)$ system. Thus, the introduction of $\check{\chi}$ is motivated by the need to rely only on the selected points as pointed out by remark \ref{rk:check}. Thus, a natural question is how the choice of local cardinal affects the quality of the approximation.

To end the theoretical analysis of \ref{alg:adaptms}, we state
\begin{theorem}
    \label{thm:localerror}
    Let $\check{\chi}^{(\ell)}_{i}$ be the cardinal function where the Lagrange conditions are enforced only in $B^{(\ell)}(\vec{x}_{i}^{(\ell)}; \varrho h_{\ell} |\log h_{\ell}|)$, with $\varrho < k \magicC$. Then, the execution algorithm \ref{alg:adaptms} using local cardinals instead of traditional cardinals, for $\ell > \ell_{0}$, produces local approximations 
    \begin{align*}
        \check{s}_{\ell} &= \sum\nolimits_{\vec{x}_{i}^{(\ell)} \in X_{\ell}\setminus \bar{X}_{\ell}} \check{e}_{\ell-1}(\vec{x}_{i}^{(\ell)}) \check{\chi}_{i} \\
        \check{e}_{\ell} &= \check{e}_{\ell-1}-\check{s}_{\ell}, \quad \check{e}_{\ell_{0}} = e_{\ell_{0}} \\
        \check{f}_{\ell} &= \check{f}_{\ell-1}+\check{s}_{\ell}, \quad \check{f}_{\ell_{0}} = f_{\ell_{0}}
    \end{align*}
    that satisfies 
    \begin{equation}\label{eq:approxcheckvsbar}
        \|\bar{s}_{\ell} - \check{s}_{\ell}\|_{L_{\infty}(\Omega)} \le (1+c_{\ref{eq:lagrangedecaysum}})^{\ell-\ell_{0}-1}\left\| e_{\ell_{0}} \right\|_{\ell^{\infty}(X_{\le \ell})} c_{\ref{eq:localcardinal-inf}} C_{\#} h_{\ell}^{\theta \varrho -2d} + c_{\ref{eq:lagrangedecaysum}}\|\check{e}_{\ell-1}-\bar{e}_{\ell-1}\|_{L_{\infty}(\Omega)}.
    \end{equation}
    Furthermore, assuming $\theta \varrho > 2d$,
    \begin{equation}\label{eq:errorcheckvsbar}
        \|\bar{e}_{\ell} - \check{e}_{\ell}\|_{L_{\infty}(\Omega)} \le c_{\ref{eq:localcardinal-inf}} C_{\#}(1+c_{\ref{eq:lagrangedecaysum}})^{\ell-\ell_{0}-1} \left\| e_{\ell_{0}} \right\|_{\ell^{\infty}(X_{\le \ell})} 2c_{h} h_{\ell_{0}+1}^{\theta \varrho -2d}
    \end{equation}
    holds.
\end{theorem}
\begin{proof}
    We have 
    \begin{align*}
       & \|\bar{s}_{\ell} - \check{s}_{\ell}\|_{L_{\infty}(\Omega)}  = \left\|\sum_{\vec{x}_{i}^{(\ell)} \in X_{\ell} \setminus \bar{X}_{\ell}} \bar{e}_{\ell-1}(\vec{x}_{i}^{(\ell)})\chi_{i} - \sum_{\vec{x}_{i}^{(\ell)} \in X_{\ell} \setminus \bar{X}_{\ell}} \check{e}_{\ell-1}(\vec{x}_{i}^{(\ell)})\check{\chi}_{i}\right\|_{L_{\infty}(\Omega)} \\
       & \le \left\|\sum_{\vec{x}_{i}^{(\ell)} \in X_{\ell} \setminus \bar{X}_{\ell}} \bar{e}_{\ell-1}(\vec{x}_{i}^{(\ell)})(\chi_{i} - \check{\chi}_{i}) \right\|_{L_{\infty}(\Omega)} + \left\| \sum_{\vec{x}_{i}^{(\ell)} \in X_{\ell} \setminus \bar{X}_{\ell}} \left(\check{e}_{\ell-1}(\vec{x}_{i}^{(\ell)})-\bar{e}_{\ell-1}(\vec{x}_{i}^{(\ell)})\right)\check{\chi}_{i}\right\|_{L_{\infty}(\Omega)} \\
        & = (I) + (II).
    \end{align*}
    We bound the norms separately. From lemma \ref{lem:localcardinal-inf} and \ref{lem:ellbound}, we have
    \begin{equation*}
        (I) = \left\|\sum_{\vec{x}_{i}^{(\ell)} \in X_{\ell} \setminus \bar{X}_{\ell}} \bar{e}_{\ell-1}(\vec{x}_{i}^{(\ell)})(\chi_{i} - \check{\chi}_{i}) \right\|_{L_{\infty}(\Omega)} \le  (1+c_{\ref{eq:lagrangedecaysum}})^{\ell-\ell_{0}-1}\left\| e_{\ell_{0}} \right\|_{\ell^{\infty}(X_{\le \ell})} c_{\ref{eq:localcardinal-inf}} h_{\ell}^{\theta \varrho -d} \check{N}(\ell).
    \end{equation*}
    (II) can be bound exploiting that $\varrho < k\magicC$, thus the local points enforced in the construction of $\check{\chi}_{i}$ are within the set $X_{\ell} \setminus \bar{X}_{\ell}$. Thus,
    \begin{align*}
        \left\| \sum_{\vec{x}_{i}^{(\ell)} \in X_{\ell} \setminus \bar{X}_{\ell}} \check{\chi}_{i}\right\|_{L_{\infty}(\Omega)} & \le \left\| \sum_{\vec{x}_{i}^{(\ell)}, \vec{x}_{j}^{(\ell)} \in X_{\ell} \setminus \bar{X}_{\ell}} B_{ij}\Phi_{\ell}(\vec{x}_{j}^{(\ell)}, \cdot)\right\|_{L_{\infty}(\Omega)} \\
        & \le \left\| \phi_{\ell}\right\|_{L_{\infty}(\Omega)} \sum_{\vec{x}_{i}^{(\ell)}, \vec{x}_{j}^{(\ell)} \in X_{\ell} \setminus \bar{X}_{\ell}} \left|B_{ij}\right|,
    \end{align*}
    from here we can apply the machinery of corollary \ref{cor:lagrangedecaysum} to get
    \begin{equation*}
        (II) = \left\| \sum_{\vec{x}_{i}^{(\ell)} \in X_{\ell} \setminus \bar{X}_{\ell}} \left(\check{e}_{\ell-1}(\vec{x}_{i}^{(\ell)})-\bar{e}_{\ell-1}(\vec{x}_{i}^{(\ell)})\right)\check{\chi}_{i}\right\|_{L_{\infty}(\Omega)} \le c_{\ref{eq:lagrangedecaysum}} \|\check{e}_{\ell-1}-\bar{e}_{\ell-1}\|_{\ell_{\infty}(X_{\ell} \setminus \bar{X}_{\ell})} .
    \end{equation*}
    Gathering the results gives
    \begin{align*}
        \|\bar{s}_{\ell} - \check{s}_{\ell}\|_{L_{\infty}(\Omega)} &\le (1+c_{\ref{eq:lagrangedecaysum}})^{\ell-\ell_{0}-1}\left\| e_{\ell_{0}} \right\|_{\ell^{\infty}(X_{\le \ell})} c_{\ref{eq:localcardinal-inf}} h_{\ell}^{\theta \varrho -d} \check{N}(\ell) + c_{\ref{eq:lagrangedecaysum}}\|\check{e}_{\ell-1}-\bar{e}_{\ell-1}\|_{\ell_{\infty}(X_{\ell} \setminus \bar{X}_{\ell})} \\
        &\le (1+c_{\ref{eq:lagrangedecaysum}})^{\ell-\ell_{0}-1}\left\| e_{\ell_{0}} \right\|_{\ell^{\infty}(X_{\le \ell})} c_{\ref{eq:localcardinal-inf}} C_{\#} h_{\ell}^{\theta \varrho -2d} + c_{\ref{eq:lagrangedecaysum}}\|\check{e}_{\ell-1}-\bar{e}_{\ell-1}\|_{L_{\infty}(\Omega)}.
    \end{align*}
    From here we have
    \begin{align*}
        \|\bar{e}_{\ell} - \check{e}_{\ell}\|_{L_{\infty}(\Omega)} &\le \|\bar{e}_{\ell-1} - \check{e}_{\ell-1} \|_{L_{\infty}(\Omega)} + \|\bar{s}_{\ell} -\check{s}_{\ell}\|_{L_{\infty}(\Omega)}\\
        & \le (1+c_{\ref{eq:lagrangedecaysum}})^{\ell-\ell_{0}-1}\left\| e_{\ell_{0}} \right\|_{\ell^{\infty}(X_{\le \ell})} c_{\ref{eq:localcardinal-inf}} C_{\#} h_{\ell}^{\theta \varrho -2d} +\\
        &+ c_{\ref{eq:lagrangedecaysum}}\|\check{e}_{\ell-1}-\bar{e}_{\ell-1}\|_{\ell_{\infty}(X_{\ell} \setminus \bar{X}_{\ell})} +\|\bar{e}_{\ell-1} - \check{e}_{\ell-1}\|_{L_{\infty}(\Omega)} \\
        & \le (1+c_{\ref{eq:lagrangedecaysum}})^{\ell-\ell_{0}-1}\left\| e_{\ell_{0}} \right\|_{\ell^{\infty}(X_{\le \ell})} c_{\ref{eq:localcardinal-inf}} C_{\#} h_{\ell}^{\theta \varrho -2d} \\
        &+ (c_{\ref{eq:lagrangedecaysum}}+1)\|\bar{e}_{\ell-1} - \check{e}_{\ell-1}\|_{L_{\infty}(\Omega)}.
    \end{align*}
    Iterating the previous equation leads to
    \begin{align*}
        \|\bar{e}_{\ell} - \check{e}_{\ell}\|_{L_{\infty}(\Omega)} &\le \sum_{j=0}^{\ell-\ell_{0}-1} (1+c_{\ref{eq:lagrangedecaysum}})^{\ell-\ell_{0}-1-j} \left\| e_{\ell_{0}} \right\|_{\ell^{\infty}(X_{\le \ell-j})} c_{\ref{eq:localcardinal-inf}} C_{\#} h_{\ell-j}^{\theta \varrho -2d}(c_{\ref{eq:lagrangedecaysum}}+1)^{j} \\
        & \le c_{\ref{eq:localcardinal-inf}} C_{\#}(1+c_{\ref{eq:lagrangedecaysum}})^{\ell-\ell_{0}-1} \left\| e_{\ell_{0}} \right\|_{\ell^{\infty}(X_{\le \ell})} \sum_{j=\ell_{0}+1}^{\ell} h_{j}^{\theta \varrho -2d}.
    \end{align*}
    Finally, the convergence of the geometric series, given $h_{j} \le h_{\ell_{0}+1}\mu^{j-\ell_{0}+1}$ along with the assumption $\theta \varrho > 2d$, let us conclude
    \begin{align*}
        \|\bar{e}_{\ell} - \check{e}_{\ell}\|_{L_{\infty}(\Omega)} &\le c_{\ref{eq:localcardinal-inf}} C_{\#}(1+c_{\ref{eq:lagrangedecaysum}})^{\ell-\ell_{0}-1} \left\| e_{\ell_{0}} \right\|_{\ell^{\infty}(X_{\le \ell})} h_{\ell_{0}+1}^{\theta \varrho -2d}(1-\mu^{\theta \varrho -2d})^{-1} \\
        & \le c_{\ref{eq:localcardinal-inf}} C_{\#}(1+c_{\ref{eq:lagrangedecaysum}})^{\ell-\ell_{0}-1} \left\| e_{\ell_{0}} \right\|_{\ell^{\infty}(X_{\le \ell})} 2c_{h} h_{\ell_{0}+1}^{\theta \varrho -2d}.
    \end{align*}
    Thus, the stated result holds.
\end{proof}

Thus, the local cardinal version of the algorithm produces a solution which is close to the approximation computed with the traditional cardinals. The quality, as we can see, similar to other results in the article depends on the adaptive turning point $\ell_{0}$ and the error of the classical scheme in such instance. Clearly, for larger values of $\ell_{0}$ or small errors $e_{\ell_{0}}$ the payoff is negligible. Additionally, we recall that with such variation, the discarded point within the adaptive selection is just required as evaluation point to update the variable $\varepsilon_{\ell}$, but are not required for any interpolant update, leading to a much more efficient algorithm.

\section{Algorithms and numerical results}
\label{sec:numerics}
In this section, we present the algorithms used in the numerical section. 
Before diving into the algorithms, we first present some general considerations on their design.
The nested sequence of sets of points $\{X_{\ell}\}_{1\le\ell \le L}$ is stored in order to improve their access in the algorithm. Indeed, we require, on each level, to have access to both the associated set of points and the residual computed on the finest set to perform the point selection. To end this, we would like, since the sets are nested, that every set $X_{\ell}$ for $1 \le \ell \le L$ is accessible from $X_{L}$. Precisely, we assume that we have access to a set $X$ and a list $\{N(\ell)\}_{1\le \ell \le L}$, such that the first $N(\ell)$ entries of $X$ coincide with the set $X_{\ell}$, using the notation $X_{\ell} := X[N(\ell)]$. 

\begin{algorithm}[ht]
    \caption{point selection algorithm}\label{alg:selection}
    \begin{algorithmic}[1]
    \Statex{Input: Point set $X$, size of the subset $N(\ell)$ of size $N(L)$, error values associated with the points $\vec{e} = \{e(\vec{x})\}_{\vec{x} \in X}$, boolean mask of removed point $\zeta^{\prime}$, threshold value $\varepsilon$ and radius $r$.}
    \Statex{Output: Updated boolean mask of removed point $\zeta^{\prime}$, selected point $\zeta$ and cardinality of the selected set $\check{N}$.}
    \State $\zeta_{hv} \leftarrow |\vec{e}| > \varepsilon$ 
    \If{$\zeta_{hv} \neq \boldsymbol{0}$}
        \For{$\vec{x}_{i} \in X_{\zeta_{hv}^{\complement}[N(\ell)] \, \texttt{xor}  \,  \zeta^{\prime}[N(\ell)]}$}
            \State $\zeta^{\prime}_{i} \leftarrow \min_{\vec{x} \in X_{\zeta_{hv}}}(\|\vec{x}-\vec{x}_{i}\|_{2}) > r$ 
        \EndFor
        \State $\zeta \leftarrow (\zeta^{\prime})^{\complement}$ 
        \For{$\vec{x} \in X_{\zeta^{\prime}}$}
            \State $\zeta_{\zeta} \leftarrow \zeta_{\zeta} \, \texttt{and} \, (\|X_{\zeta}-\vec{x}\|_{2} > r)$ 
        \EndFor
        \State $\check{N} \leftarrow \texttt{sum}(\zeta)$
    \Else
        \State $\check{N} \leftarrow 0$
        \State $\zeta^{\prime} \leftarrow \boldsymbol{1}_{N(L)}$
        \State $\zeta \leftarrow \boldsymbol{0}_{N(\ell)}$
    \EndIf
    \end{algorithmic}
\end{algorithm}

We first present the selection algorithm. It deals with masks, i.e., boolean vectors; instead of storing the subset of desired points, one can recover the subset of interest by application of the mask.
Since the aim of the algorithm is to perform the selection \ref{eq:adaptpoint}, we need just the finest set $X$, from which we can recover both the set of points that we employ for the criteria validation, and the set on which the selection is performed. Indeed, $N(\ell)$ allows us to have access to the latter. Associated with the full set of points $X$, we have the set of error evaluation $\vec{e}$, needed for thresholding. Lastly, we have the thresholding parameter $\varepsilon$ and the influence radius $r$. 
In the following, we use the notation $X_{\zeta}$ for the subset of $X$ associated with the mask of boolean indices $\zeta$.
Algorithm \ref{alg:selection} is structured as follows: first, we compute the mask of the point whose associated error is higher than the threshold $\zeta_{hv}$. Here, if the mask is not the null mask, we proceed, otherwise the error is lower than the threshold in the finest set, therefore, there is nothing to refine; the algorithm terminates returning the null mask for the selection the identity mask for the removed points and 0 as cardinality for the selected points. On the other hand, if at least one point has error greater than $\varepsilon$, we compute the mask of the removed points, i.e., the points whose are at least $r$-far from the points with large error $Y_{\zeta_{hv}}$. Theorem \ref{thm:consistency} allows us to consider points that were not already removed in previous levels, and Lemma \ref{lem:checkset} suggests which points are meaningful to inspect, i.e., the points whose error is smaller than the threshold. This process is done iteratively on all points of $X$.
Having the mask of the removed points $\zeta^{\prime}$, we proceed with the computation of the mask of the selected points. We take the complement of the mask of the removed points $(\zeta^{\prime})^{\complement}$ as starting selection. Lastly, we iterate on the removed points with the goal of removing all points close to them, updating the mask of the selected points $\zeta$. Note that we use the notation $\zeta_{\zeta}$ to work with the subset of indices that are feasible candidates, which stays so, thanks to the bitwise and operator, just if they are $r$-far from the current removed point $\vec{x}$ in the loop.
Finally, we compute the number of selected points exploiting the boolean nature of the mask.
The algorithm returns $\zeta^{\prime} \in \{0,1\}^{N(L)}$, $\zeta \in \{0,1\}^{N(\ell)}$ and $\check{N} \in \N$.

\begin{algorithm}[ht]
    \caption{Adaptive multiscale approximation algorithm}\label{alg:adapt}
    \begin{algorithmic}[1]
    \Statex{Input: Number of levels $L \in \N$, multi-level $d$-array of points $X$ with size vector $\{N(\ell)\}_{1\le \ell \le L}$, values associated with the points $\boldsymbol{f} =\{f(\vec{x})\}_{\vec{x} \in X} \in \R^{N(L)}$. Constants $k$, $\nu$, $\mu$. Set of evaluation points $Y$. starting mesh norm $h_{1}$, radii constant $\magicC$, starting threshold $\varepsilon_{0}$ and cg tolerance $\epsilon_{tol}$}
    \Statex{Output: Level-wise approximation $f_{\ell}$ computed on $Y$, for $1 \le \ell \le L$.}
    \State $f_{0} \leftarrow \boldsymbol{0}_{Y}$
    \State $\zeta^{\prime} \leftarrow \boldsymbol{0}_{N(L)}$
    \For{$1 \le \ell \le L$}
        \If{$\ell > L^{\star}$}     
            \State $\zeta^{\prime}, \zeta, \check{N} \leftarrow \texttt{PointSelection}(X,N(\ell),\boldsymbol{f},\zeta^{\prime},\varepsilon_{\ell-1},k\magicC\mu^{\ell-1} h_{1}|\log(\mu^{\ell-1} h_{1})|)$
        \Else   
            \State $\zeta, \check{N} \leftarrow \boldsymbol{1}_{N(L)}, N(L) $
        \EndIf
        \State $\delta_{\ell} \leftarrow \nu \mu^{\ell-1} h_{1}$
        \State $K := \Phi_{\ell}(\vec{x}_{i}^{(\ell)}, \vec{x}_{j}^{(\ell)})  \quad 1 \le i,j \le N(\ell)$
        \State $K^{-1} \leftarrow \texttt{cg}(K, \text{id}_{N(\ell)}, \epsilon_{tol})$
        \State $P := \Phi_{\ell}(\vec{x}_{i}^{(\ell)}, \vec{x}_{j}^{(\ell)})  \quad 1 \le i \le N(\ell) < j \le N$
        \State $\chi = K^{-1} P$
        \State $s_{\ell} [N(\ell)] \leftarrow \boldsymbol{f}[N(\ell)]$
        \State $(s_{\ell})[N(\ell)]^{\complement} \leftarrow \chi_{\zeta}^{T} \boldsymbol{f}[N(\ell)]_{\zeta}$
        \If{$\ell \ge L^{\star}$}
            \State $\zeta_{\varepsilon} \leftarrow \boldsymbol{0}_{N(L)}$
            \For{$\vec{x} \in (X_{\ell})_{\zeta^{\prime}}$}
                \State $(\zeta_{\varepsilon})_{\zeta_{\varepsilon}^{\complement}} \leftarrow \|X_{\zeta_{\varepsilon}^{\complement}} - \vec{x}\|_{2} < \magicC\mu^{\ell-1} h_{1}|\log(\mu^{\ell-1} h_{1})|)$
            \EndFor
            \State $\varepsilon_{\ell} \leftarrow \varepsilon_{\ell-1} + \max({s_{\ell}}_{\zeta_{L}})$
        \EndIf
        \State $f_{\ell} \leftarrow  f_{\ell-1} +  \chi_{\zeta}^{T} \boldsymbol{f}[N(\ell)]_{\zeta}$
        \State $\boldsymbol{f} \leftarrow  \boldsymbol{f} - s_{\ell}$
    \EndFor
    \end{algorithmic}
\end{algorithm}

The adapt multiscale algorithm \ref{alg:adapt} is the main focus of the manuscript. We assume to start with a sequence of nested sets of points such that Assumption \ref{ass:pointset} holds with respect to constants $\mu$, $\nu$, $h_{1}$ and $k$ stored as previously mentioned. Then the target function $f$ values associated with the sets of points. We also consider the function $\Phi$ such that Assumption \ref{ass:kernel} holds. Lastly, we introduce parameters $\magicC$ on the selection influence, $\varepsilon_{0}$ the initial selection threshold \eqref{eq:adaptpoint} and the tolerance of the conjugate gradient to solve the linear systems $\varepsilon_{tol}$. Additionally, we have the set where the approximation is computed.
The algorithm works as outlined in \ref{alg:adaptms}; iteratively, for levels up to $L^{\star}$ we do not compute the selection and assign the trivial mask as selection mask. On the other hand, from level $L^{\star}$ we perform the selection on the current set. Then, we compute the interpolation matrix and approximate the coefficients of the cardinals using the conjugate gradient with tolerance $\varepsilon_{0}$, due to have the matrix which is symmetric and positive definite. Next, the threshold parameter is updated; first the mask of points taken into account by definition \ref{def:eps} is computed, then the update is performed. Lastly, the evaluation of the approximation on the evaluation set is computed.
Is worth to point out that we leverage the nested hypothesis of the sets of points to spare computations. Indeed, all updates applied to the finest level are by construction applied to all previous levels. Additionally, the consistency from theorem \ref{thm:consistency} is guaranteed by the application of selection from level $L^{\star}$.

As last algorithm, we present the adaptive multiscale method with local cardinal function. Indeed, introducing $\varrho < k \magicC$, we achieve adaptivity where discarded points are not required even in the cardinal construction.

\begin{algorithm}[ht]
    \caption{Adaptive multiscale approximation algorithm, local variation}\label{alg:loc}
    \begin{algorithmic}[1]
    \Statex{Input: Number of levels $L \in \N$, multi-level $d$-array of points $X$ with size vector $\{N(\ell)\}_{1\le \ell \le L}$, values associated with the points $\boldsymbol{f} =\{f(\vec{x})\}_{\vec{x} \in X} \in \R^{N(L)}$. Constants $k$, $\nu$, $\mu$. Set of evaluation points $Y$. starting mesh norm $h_{1}$, radii constant $\magicC$, locality radii $\varrho$, starting threshold $\varepsilon_{0}$ and cg tolerance $\epsilon_{tol}$}
    \Statex{Output: Level-wise approximation $f_{\ell}$ computed on $Y$, for $1 \le \ell \le L$.}
    \State $f_{0} \leftarrow \boldsymbol{0}_{Y}$
    \State $\zeta^{\prime} \leftarrow \boldsymbol{0}_{N(L)}$
    \For{$1 \le \ell \le L$}
        \State $\delta_{\ell} \leftarrow \nu \mu^{\ell-1} h_{1}$
        \If{$\ell > L^{\star}$}     
            \State $\zeta^{\prime}, \zeta, \check{N} \leftarrow \texttt{PointSelection}(X,N(\ell),\boldsymbol{f},\zeta^{\prime},\varepsilon_{\ell-1},k\magicC\mu^{\ell-1} h_{1}|\log(\mu^{\ell-1} h_{1})|)$
            \State $K^{-1} \leftarrow \boldsymbol{0}_{N(\ell) \times N(\ell)}$
            \For{$\vec{x}_{i}^{(\ell)} \in (X_{\ell})_{\zeta}$}
                \State $\zeta_{i} \leftarrow \|\vec{x}_{i}^{(\ell)}-X_{\ell}\|_{2} < \varrho \mu^{\ell-1} h_{1}|\log(\mu^{\ell-1} h_{1})|$
                \State $K_{i} := \Phi_{\ell}(\vec{x}_{i}^{(\ell)}, \vec{x}_{j}^{(\ell)})  \quad \vec{x}_{i}^{(\ell)}, \vec{x}_{j}^{(\ell)} \in (X_{\ell})_{\zeta_{i}} $
                \State $(K^{-1})_{i, \zeta_{i}} \leftarrow \texttt{cg}(K_{i}, \text{id}_{N(\ell)_{\zeta_{i}}}, \epsilon_{tol})$
            \EndFor
        \Else   
            \State $\zeta, \check{N} \leftarrow \boldsymbol{1}_{N(L)}, N(L)$
            \State $K := \Phi_{\ell}(\vec{x}_{i}^{(\ell)}, \vec{x}_{j}^{(\ell)})  \quad 1 \le i,j \le N(\ell)$
            \State $K^{-1} \leftarrow \texttt{cg}(K, \text{id}_{N(\ell)}, \epsilon_{tol})$
        \EndIf
        \State $P := \Phi_{\ell}(\vec{x}_{i}^{(\ell)}, \vec{x}_{j}^{(\ell)})  \quad 1 \le i \le N(\ell) < j \le N$
        \State $\chi = K^{-1} P$
        \State $s_{\ell} [N(\ell)] \leftarrow \boldsymbol{f}[N(\ell)]$
        \State $(s_{\ell})[N(\ell)]^{\complement} \leftarrow \chi_{\zeta}^{T} \boldsymbol{f}[N(\ell)]_{\zeta}$
        \If{$\ell \ge L^{\star}$}
            \State $\zeta_{\varepsilon} \leftarrow \boldsymbol{0}_{N(L)}$
            \For{$\vec{x} \in (X_{\ell})_{\zeta^{\prime}}$}
                \State $(\zeta_{\varepsilon})_{\zeta_{\varepsilon}^{\complement}} \leftarrow \|X_{\zeta_{\varepsilon}^{\complement}} - \vec{x}\|_{2} < \magicC\mu^{\ell-1} h_{1}|\log(\mu^{\ell-1} h_{1})|)$
            \EndFor
            \State $\varepsilon_{\ell} \leftarrow \varepsilon_{\ell-1} + \max({s_{\ell}}_{\zeta_{L}})$
        \EndIf
        \State $f_{\ell} \leftarrow  f_{\ell-1} +  \chi_{\zeta}^{T} \boldsymbol{f}[N(\ell)]_{\zeta}$
        \State $\boldsymbol{f} \leftarrow  \boldsymbol{f} - s_{\ell}$
    \EndFor
    \end{algorithmic}
\end{algorithm}

The main difference between algorithms \ref{alg:adapt} and \ref{alg:loc} lies in the computation of the coefficients of the cardinals. Indeed, while in the former we compute the full $N(\ell) \times N(\ell)$ system is solved, in the latter we deal with $\check{N}$ $N(\ell)_{\zeta_{i}} \times N(\ell)_{\zeta_{i}}$. Precisely, in algorithm \ref{alg:loc}, for every point $\vec{x}_{i}^{(\ell)}$ we compute the coefficient of $\check{\chi_{i}}$ in the local subsystem of its neighbourhood, then the coefficient are stored with respect to the point indices in order to evaluate all cardinals simultaneously with a matrix-vector multiplication.

\section{Numerical results}

In this section we look at the numerical evidence of our theoretical results. 
All nested sequences of sets of points considered are generated starting from equispaced points in the interval $[0, 10]$ with fill distance $h_{1} := \mu/\sqrt{d}$; further levels are obtained taking equispaced points between adjacent points such that $h_{\ell} = h_{1}\mu^{\ell-1}$. Eventually, for experiment with $d>1$, a $d$-dimensional grid is built. Details regarding sets of points related to different choices of $\mu$ are detailed on table \ref{tab:datasets}, while table \ref{tab:Lstar} shows the associated $L^{\star}$ values. Lastly, table \ref{tab:params} shows the parameters of algorithm \ref{alg:adapt} fixed over all numerical experiments presented in this section \cite{lotfadapt}.

\begin{table}[ht]
    \centering
    \begin{tabular}{c|c|c|c|c}
        $\nu$ & $\varepsilon_{0}$ & $k$ & $\magicC$ & $\varepsilon_{tol}$ \\
         \hline         
         4 & 1e-8 & 2 & 2 & 1e-8 
    \end{tabular}
    \caption{Values of the parameters of algorithm \ref{alg:adapt} for the experiments.}
    \label{tab:params}
\end{table}

\begin{table}[ht]
    \centering
    \begin{tabular}{c|c|c|c|c}
         & $X_{1}$ & $X_{2}$ & $\dots$ & $X_{L}$ \\
         \hline
        $\mu = 0.40$ & 26 & 51 & $\dots$ &  12445 (L=8) \\
        $\mu = 0.30$ & 29 & 57 & $\dots$ &  14633 (L=7) \\
        $\mu = 0.30$ & 34 & 100 & $\dots$ & 11698 (L=6) \\
        $\mu = 0.20$ & 40 & 126 & $\dots$ &  8205 (L=5) \\
        $\mu = 0.20$ & 50 & 235 & $\dots$ &  6015 (L=4) \\
        $\mu = 0.15$ & 67 & 401 & $\dots$ & 17867 (L=4) \\
        $\mu = 0.10$ & 101 & 945 & $\dots$ & 9334 (L=3) \\
    \end{tabular}
    \caption{Number of points for sets $X_{\ell}$ with respect to $\mu$ generated in $[0, 10]$.}
    \label{tab:datasets}
\end{table}

\begin{table}[ht]
    \centering
    \begin{tabular}{c|c|c|c|c|c|c|c}
        $\mu$ & $0.10$ & $0.15$ & $0.20$ & $0.25$ & $0.30$ & $0.35$ & $0.40$ \\
         \hline         
        $L^{\star}$ & 1 & 1 & 1 & 1 & 2 & 3 & 4
    \end{tabular}
    \caption{Values of $L^{\star}$ for the different datasets introduced in table \ref{tab:datasets}.}
    \label{tab:Lstar}
\end{table}

With this setup, we show the potential of this strategy and how can be exploited in a more general framework.
We approximate the function 

\begin{equation*}
    f_{c, r}(x) := 
    \begin{cases}
        e^{r^{-2}-\frac{1}{r^{2}-(x-c)^{2}}} & \forall x \in [c-r, c+r] \\
        0 & \text{otherwise}
    \end{cases}, 
\end{equation*}

which is $\mathcal{C^{\infty}}$ in $\R$. 

We start by taking into account $f_{5, 0.03}$ in the interval $[0, 10]$. The first aspect of the adaptivity scheme that we investigate is the thresholding strategy. Indeed, lemma \ref{lem:epsbound} provides a bound over the threshold parameter $\varepsilon_{\ell}$ with respect to the scheme parameters. In particular, we expect $\varepsilon_{\ell}$ to be controlled for smaller values of $\mu$, i.e., $\varepsilon_{\ell} \xrightarrow{\, \mu \rightarrow 0 \,} \varepsilon_{0}$, since $h_{1}$ is proportional to $\mu$. Additionally, we expect a similar behaviour for $\magicC \rightarrow \infty$. We first investigate such behaviour, with a pure adaptive scheme, i.e., forcing $L^{\star} = 0$ and refining every set of points.
Tables \ref{tab:threshold} and \ref{tab:magicC}, highlight that as $\mu$ decrease and $\magicC$ increase the threshold decrease towards $\varepsilon_{0} = 1e-8$. It is worth to mention that $\varepsilon_{0} = 1e-8$, and for $\mu$ not small enough or $\magicC$ not large enough the accuracy may drop significantly. However, the same experiment performed without enforcing the value of $L^{\star}$ leads to threshold values bounded by $1e-7$ for all choices of $\mu$ and $\magicC$.

The fact that the adaptivity should not be introduced in the early stages of the scheme is common sense, however these results highlighted the possible impact on early adaptivity in the scheme. The following results are obtained through algorithms \ref{alg:adapt} and \ref{alg:loc}, without enforcing different values on $L^{\star}$.

\begin{table}[ht]
    \centering
    \caption{Threshold value $\varepsilon_{\ell}$ on levels $\ell=1,2,3$ for different choices of $\mu$.}
    \label{tab:threshold}
    \begin{tabular}{c|ccc}
        \hline
        $\mu$ & $\ell=1$ & $\ell=2$ & $\ell=3$ \\
        \hline
        0.40 & $6.6995e-3$ & $6.9860e-3$ & $1.1891e-2$ \\
        0.35 & $6.8701e-3$ & $7.1566e-3$ & $1.2479e-2$ \\
        0.30 & $1.0922e-4$ & $1.3830e-4$ & $1.7669e-4$ \\
        0.25 & $4.3453e-4$ & $5.8628e-4$ & $9.3406e-4$ \\
        0.20 & $5.1031e-4$ & $5.5866e-4$ & $9.5554e-4$ \\
        0.15 & $7.8425e-4$ & $1.0550e-3$ & $1.6883e-3$ \\
        0.10 & $2.1325e-5$ & $2.8806e-5$ & $4.5135e-5$ \\
        \hline
    \end{tabular}
\end{table}

\begin{table}[ht]
    \centering
    \setlength{\tabcolsep}{3pt}
    \caption{Threshold value $\varepsilon_{\ell}$ on levels $\ell=1,\ldots,7$ for different choices of $\magicC$ with $\mu=0.4$.}
    \label{tab:magicC}
    \begin{tabular}{c|ccccccc}
        \hline
        $\magicC$ & $\ell=1$ & $\ell=2$ & $\ell=3$ & $\ell=4$ & $\ell=5$ & $\ell=6$ & $\ell=7$\\
        \hline
        1 & $5.030e-6$ & $1.228e-5$ & $1.476e-5$ & $1.774e-5$ & $2.594e-5$ & $4.335e-5$ & $5.787e-5$ \\
        2 & $5.030e-6$ & $5.259e-6$ & $5.783e-6$ & $8.641e-6$ & $1.153e-5$ & $1.634e-5$ & $2.779e-5$ \\
        3 & $6.285e-7$ & $6.541e-7$ & $8.667e-7$ & $1.084e-6$ & $1.302e-6$ & $2.384e-6$ & $3.463e-6$ \\
        4 & $6.285e-7$ & $6.541e-7$ & $8.667e-7$ & $1.084e-6$ & $1.302e-6$ & $1.937e-6$ & $3.016e-6$ \\
        5 & $8.620e-8$ & $8.620e-8$ & $8.904e-8$ & $1.144e-7$ & $1.887e-7$ & $2.711e-7$ & $4.665e-7$ \\
        6 & $8.269e-8$ & $8.269e-8$ & $8.553e-8$ & $1.109e-7$ & $1.454e-7$ & $1.854e-7$ & $2.555e-7$ \\
        \hline
    \end{tabular}
\end{table}

Table \ref{tab:ratio} shows the ratio of $\check{X}_{\ell} \setminus X_{\ell}$ over different levels for different choices of $\mu$, where the values for $\mu = 0.4$ are also plotted. As we can see, pairing the results with table \ref{tab:Lstar}, except for $\mu = 0.25$, already the first adaptive levels reduce the point required and the ratio of $\check{X}_{\ell} \setminus X_{\ell}$ drops consistently.
Indeed, as suggested in theorem \ref{thm:compactsupp}, we have that the first adaptive iteration deals with a large support which leads to a higher percentage of selected points. The trend start to revert once the error outside the support of the original function, introduced by previous iterations, is smaller than the threshold parameter, and thus when this holds into a large neighborhood, such points are removed leading to a decay in the percentage of points used. 

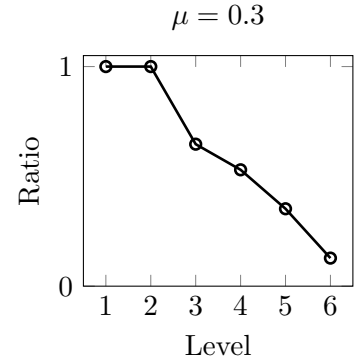
\begin{table}[ht]
\centering
\caption{Ratio of $\check{X}_{\ell} \setminus X_{\ell}$ over different levels.}
\label{tab:ratio}
\begin{minipage}{0.60\textwidth}
    \centering
    \small
    \setlength{\tabcolsep}{2pt}
        \begin{tabular}{c|ccccccc}
        \diagbox{$\ell$}{$\mu$} & $0.40$ & $0.35$ & $0.30$ & $0.25$ & $0.20$ & $0.15$ & $0.10$ \\
        \hline
        1 & 1.0    & 1.0    & 1.0    & 1.0    & 1.0    & 1.0    & 1.0    \\
        2 & 1.0    & 1.0    & 1.0    & 1.0    & 0.7248 & 0.5311 & 0.3164 \\
        3 & 1.0    & 1.0    & 0.6466 & 0.6848 & 0.5838 & 0.4255 & 0.2323 \\
        4 & 1.0    & 0.8643 & 0.5299 & 0.4830 & 0.4173 & 0.2990 & --     \\
        5 & 0.2874 & 0.6664 & 0.3525 & 0.2569 & --     & --     & --     \\
        6 & 0.2137 & 0.3842 & 0.1278 & --     & --     & --     & --     \\
        7 & 0.1720 & 0.1975 & --     & --     & --     & --     & --     \\
        8 & 0.0771 & --     & --     & --     & --     & --     & --     \\
        \hline
    \end{tabular}
\end{minipage}
\hfill
\begin{minipage}{0.35\textwidth}
    \centering
    \begin{tikzpicture}
        \begin{axis}[
            title={$\mu = 0.3$},
            width=\linewidth,
            height=0.9\linewidth,
            xlabel=Level,
            ylabel={Ratio},
            ymin=0,
            ymax=1.05,
            xtick={1,...,6},
            ytick={0,1},
        ]
        \addplot[
            black,
            mark=o,
            line width=1pt
        ] coordinates {
            (1,1.0)
            (2,1.0)
            (3,0.6466)
            (4,0.5299)
            (5,0.3525)
            (6,0.1278)
        };
        \end{axis}
    \end{tikzpicture}
\end{minipage}
\end{table}

It is clear that the domain influences the performance of the adaptive scheme. In our example, the support of $f_{5, 0.03}$ might appear small compared to the interval $[0,10]$. However, one should take into account that after the first iterations, which is always not adaptive, the error that we approximate in subsequent iterations is not localized into a neighbourhood of the original support.
Though the consistency of theorem \ref{thm:consistency} is respected independently with respect to the domain, as can be seen in figure \ref{fig:comparison}, the rate of $\check{X}_{\ell} \setminus X_{\ell}$ depends on the domain. In the approximation of $f_{5, 0.03}$ with $\mu = 0.3$, is sufficient to restrict the domain to $[2,8]$ to move the first level with a ratio different from one from $\ell=3$ to $\ell = 6$. On the other hand, figure \ref{fig:comparison} also shows how the interpolation error behaves much better in the smaller domain, which clearly is due to the higher number of points involved in the approximation not on to the domain itself.

\begin{figure}[ht] 
    \centering
    \begin{subfigure}{0.48\linewidth}
        \centering
        \begin{tikzpicture}[scale=0.8]
            \begin{axis}[
                ymin=-5e-5,
                ymax=5e-5
            ]
            \addplot[
                no marks,
            ]
            table[
                col sep=comma,
                x index = 0,
                y index = 1
            ]{error4on10-0.3.csv};    
            \addplot[only marks,
                     mark=*, 
                     ]
            table[
                col sep=comma,
                x index=0,
                y expr=-4.5e-5,
            ]{selpoint4on10-0.3.csv};
            \end{axis}
        \end{tikzpicture}
        \caption{On domain $[0, 10]$;}
    \end{subfigure}
    \hfill
    \begin{subfigure}{0.48\linewidth}
        \centering
        \begin{tikzpicture}[scale=0.8]
            \begin{axis}[
                ymin=-8e-6,
                ymax=8e-6
            ]
            \addplot[
                no marks,
            ]
            table[
                col sep=comma,
                x index = 0,
                y index = 1
            ]{error4on5-0.3.csv};
            \addplot[only marks, mark=*]
            table[
                col sep=comma,
                x index=0,
                y expr=-6.5e-6
            ]{selpoint4on5-0.3.csv};
            \end{axis}
        \end{tikzpicture}
        \caption{On domain $[2, 8]$;}
    \end{subfigure}
    \caption{Interpolation error $\bar{e}_{6}$ and selected point $\check{X}_{6}$ for different configuration of points. We considered the function $f_{5, 0.03}$ on domain $[2,8]$ and $[0,10]$. The points were generated based on the interval and $\mu = 0.3$.}
    \label{fig:comparison}
\end{figure}
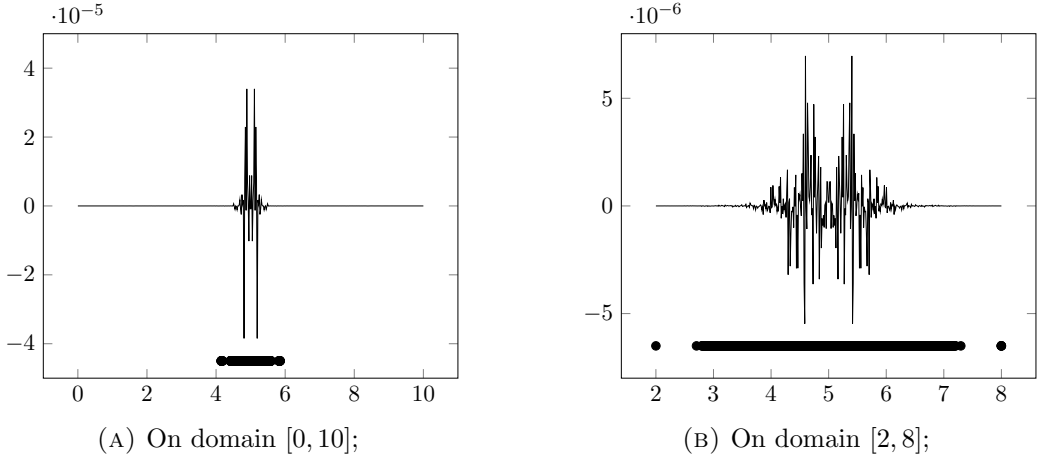

Lastly, we sought numerical evidence for theorem \ref{thm:localerror}. We computed the error on $500$ points in the interval $[0,10]$ with algorithms \ref{alg:adapt} and \ref{alg:loc} for different values of $\varrho$ for fixed $\mu$ and for different values of $\mu$ for fixed $\varrho$ and collected the difference in tables \ref{tab:rho} and \ref{tab:error}. Indeed, for fixed $\mu$ the error drops as $\varrho$ increases or $\mu$ decreases for small values of $\mu$. Clearly, errors smaller than $1e-8$ are less meaningful as precise values, given our parameters choices, however, are still a source of insight especially if compared to errors of closer level within the same experimental setup.

\begin{table}[ht]
    \centering
    \begin{tabular}{c|cccc}
        \hline
        Level & $\varrho=1$ & $\varrho=2$ & $\varrho=3$ & $\varrho=4$ \\
        \hline
        1 & $0.0$ & $0.0$ & $0.0$ & $0.0$ \\
        2 & $0.0$ & $0.0$ & $0.0$ & $0.0$ \\
        3 & $1.3662e-04$ & $8.3405e-07$ & $3.5147e-09$ & $6.7670e-10$ \\
        4 & $1.1459e-04$ & $7.3824e-07$ & $3.0791e-09$ & $4.0970e-10$ \\
        5 & $9.4496e-05$ & $5.8438e-07$ & $3.0888e-09$ & $2.4866e-10$ \\
        6 & $2.6828e-05$ & $1.4013e-07$ & $2.9993e-09$ & $2.4151e-10$ \\
        \hline
    \end{tabular}
    \caption{Norm of the error difference using local with respect to traditional cardinal functions. Results are computed using Algorithms \ref{alg:adapt} and \ref{alg:loc}, with $\mu = 0.3$ and different values of $\varrho$.}
    \label{tab:rho}
\end{table}

\begin{table}[ht]
    \centering
    \begin{tabular}{c|ccccccc}
        \diagbox{$\ell$}{$\mu$} & $0.40$ & $0.35$ & $0.30$ & $0.25$ & $0.20$ & $0.15$ & $0.10$ \\
        \hline
        1 & 0.0      & 0.0       & 0.0      & 0.0      & 0.0      & 0.0      & 0.0      \\
        2 & 0.0      & 0.0       & 0.0      & 5.652e-4 & 2.176e-4 & 1.039e-5 & 1.416e-7 \\
        3 & 0.0      & 0.0       & 8.340e-7 & 2.746e-4 & 5.914e-5 & 9.729e-6 & 5.758e-9 \\
        4 & 0.0      & 1.339e-10 & 7.382e-7 & 2.097e-4 & 1.243e-6 & 1.711e-6 &   --     \\
        5 & 1.346e-8 & 4.741e-9  & 5.843e-7 & 1.867e-5 &   --     &   --     &   --     \\
        6 & 1.439e-8 & 2.701e-9  & 1.401e-7 &   --     &   --     &   --     &   --     \\
        7 & 4.987e-9 & 1.614e-9  &   --     &   --     &   --     &   --     &   --     \\
        8 & 4.987e-9 &  --       &   --     &   --     &   --     &   --     &   --     \\
        \hline
    \end{tabular}
    \caption{Norm of the error difference using local with respect to traditional cardinal functions. Results are computed using Algorithms \ref{alg:adapt} and \ref{alg:loc}, with different values of $\mu$ and $\varrho = 2$.}
    \label{tab:error}
\end{table}

Finally, previous results were presented in low dimension in order to afford large domains and a high number of points to highlight the features of the scheme. Indeed, we computed the approximation of

\begin{equation*}
    f_{2} := \begin{cases}
        e^{\frac{100}{9}-\frac{1}{0.09-\|x-(1.5, 1.5)\|_{2}^{2}}} \quad &\forall \, x \in B((1.5, 1.5); 0.3) \\
        0 & \text{otherwise}
    \end{cases} ,
\end{equation*}

over a nested set with $\mu = 0.25$, $h_{\ell} = \mu^{\ell}$ and $L^{\star} = 1$ in the domain $[0, 3]^{2}$, counting $37249$ points, using algorithm \ref{alg:adapt}.
Figure \ref{fig:2d} shows that dimension is not a restriction; in fact, we observe an analogue behaviour to the 1-dimensional setting. However, it should be taken in consideration that naive approximation in high dimension might be expensive from a storage prospective. As in other circumstances, this obstacle can be overcome with appropriate techniques as sparse grids.

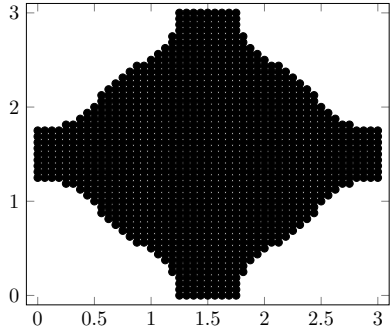
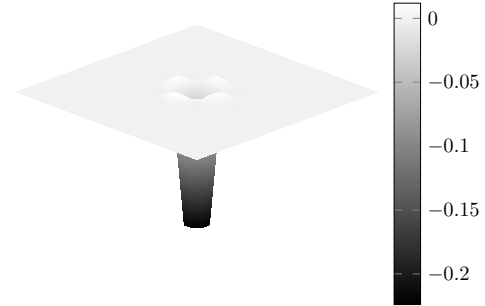
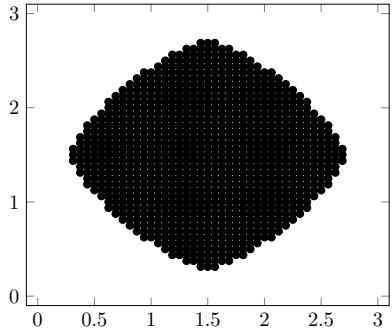
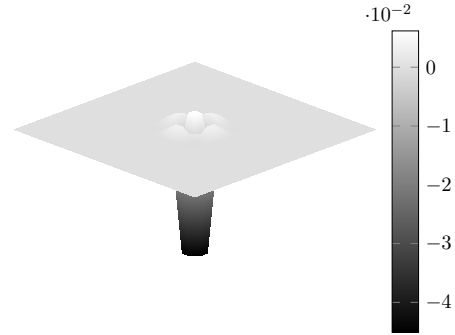
\begin{figure}[ht] 
    \centering
    \begin{subfigure}{0.48\linewidth}
        \centering
        \begin{tikzpicture}[scale=0.7]
            \begin{axis}[
                xmin = -0.1,
                xmax = 3.1,
                ymin = -0.1,
                ymax = 3.1,]
            \addplot[
                only marks,
            ]
            table[
                col sep=comma,
                x index = 0,
                y index = 1
            ]{points2d-2.csv};
            \end{axis}
        \end{tikzpicture}
        \caption{Selected points on level $2$.}
    \end{subfigure}
    \hfill
    \begin{subfigure}{0.48\linewidth}
        \centering
        \centering
    \begin{tikzpicture}[scale=0.7]
    	\begin{axis}[
            axis lines=none,
            ticks=none,
            mesh/rows=30,
            view={45}{30},
    		colorbar,
            colormap/blackwhite,
            colorbar style={
            ytick={0.0,-0.05,-0.10,-0.15,-0.20},
            yticklabel style={
            /pgf/number format/.cd,
            fixed,
            precision = 2
            }
            }
            ]
    	\addplot3[surf, shader=interp,] 
    		table[
                col sep=comma,
                x index = 0,
                y index = 1,
                z index = 2
            ] {error2d-2.csv};
    	\end{axis}
    \end{tikzpicture}
        \caption{Approximation error on level $2$.}
    \end{subfigure}

    \vspace{0.5cm}

    \begin{subfigure}{0.48\linewidth}
        \centering
        \begin{tikzpicture}[scale=0.7]
            \begin{axis}[
                xmin = -0.1,
                xmax = 3.1,
                ymin = -0.1,
                ymax = 3.1,]
            \addplot[
                only marks,
            ]
            table[
                col sep=comma,
                x index = 0,
                y index = 1
            ]{points2d-3.csv};
            \end{axis}
        \end{tikzpicture}
        \caption{Selected points on level $3$.}
    \end{subfigure}
    \hfill
    \begin{subfigure}{0.48\linewidth}
        \centering
        \begin{tikzpicture}[scale=0.7]
            \begin{axis}[
            axis lines=none,
            ticks=none,
            mesh/rows=30,
            view={45}{30},
    		colorbar,
            colormap/blackwhite,
            ]
    	\addplot3[surf, shader=interp,] 
    		table[
                col sep=comma,
                x index = 0,
                y index = 1,
                z index = 2
            ] {error2d-3.csv};
    	\end{axis}
        \end{tikzpicture}
        \caption{Approximation error on level $3$.}
    \end{subfigure}
    \caption{Interpolation error $e_{\ell}$ and selected point $\check{X}_{\ell}$ in the approximation process of $f_{2}$ on $[0,3]^{2}$ with algorithm \ref{alg:adapt} for $\ell = 2,3$.}
    \label{fig:2d}
\end{figure}



\end{document}